\documentclass[leqno,11pt]{siamltex}
\usepackage{amsfonts}
\usepackage{amsmath}
\usepackage{amssymb}
\usepackage{subfig}
\usepackage{graphicx,version,color}
\usepackage{bm}
\usepackage{xcolor}
\usepackage{multirow}
\usepackage{booktabs}

\catcode`\@=11
\@addtoreset{equation}{section} 
\newcommand{\tc}{{\rm c}}
\newcommand{\td}{{\rm d}}

\newcommand{\ti}{{\rm i}}
\newcommand{\tP}{{\rm P}}
\newcommand{\tH}{{\rm H}}
\newcommand{\ext}{{\rm ext}}
\newcommand{\he}{\hat{e}}
\newcommand{\hp}{\hat{p}}
\newcommand{\hu}{\hat{u}}
\newcommand{\hw}{\hat{w}}
\newcommand{\hphi}{\hat{\phi}}
\newcommand{\hpsi}{\hat{\psi}}
\newcommand{\tdu}{\tilde{u}}
\newcommand{\tdp}{\tilde{p}}

\newtheorem{remark}{Remark}[section]

\newtheorem{prop}{Proposition}[section]

\title{Deep adaptive basis Galerkin method for high-dimensional evolution equations with oscillatory solutions}
\author{Yiqi Gu\thanks{Department of Mathematics, The University of Hong Kong, Pokfulam, Hong Kong ({\tt yiqigu@hku.hk, mng@maths.hku.hk}). This work is supported by HKRGC GRF 12300519, 17201020 and 17300021,
C1013-21GF, C7004-21GF and Joint NSFC-RGC N-HKU76921.} \and Micheal K. Ng\footnotemark[1]}

\begin{document}

\maketitle

\begin{abstract}
In this paper, we study deep neural networks (DNNs) for solving high-dimensional evolution equations with oscillatory solutions. Different from deep least-squares methods that deal with time and space variables simultaneously, we propose a deep adaptive basis Galerkin (DABG) method, which employs the spectral-Galerkin method for the time variable of oscillatory solutions and the deep neural network method for high-dimensional space variables. The proposed method can lead to a linear system of differential equations having unknown DNNs that can be trained via the loss function. We establish a posterior estimates of the solution error, which is bounded by the minimal loss function and the term $O(N^{-m})$, where $N$ is the number of basis functions and $m$ characterizes the regularity of the e'quation. We also show that if the true solution is a Barron-type function, the error bound converges to zero as $M=O(N^p)$ approaches to infinity, where $M$ is the width of the used networks, and $p$ is a positive constant. Numerical examples, including high-dimensional linear evolution equations and the nonlinear Allen-Cahn equation, are presented to demonstrate the performance of the proposed DABG method is better than that of existing DNNs.
\end{abstract}


\pagestyle{myheadings}
\thispagestyle{plain}
\markboth{}{}

\section{Introduction}\label{Sec_introduction}
The partial differential equations (PDEs) that involve time widely appear in areas such as fluid dynamics, biological population and chemical diffusion. Numerical methods for solving evolution equations have been developed extensively in the past several decades. In most cases, the equations are discretized in time by numerical schemes (e.g., linear multistep methods and Runge-Kutta methods), and the derived semi-discrete equations in space are thereafter tackled by traditional PDE solvers such as finite difference methods and finite element methods. Despite high accuracy, these traditional solvers suffer from the curse of dimensionality in storage and hence generally incapable of problems with very high dimensions.

In recent years, a series of methods based on deep neural networks (DNNs) have been proposed to solve high-dimensional PDEs (e.g., \cite{E2017,E2018,Han2018,Khoo2021,Sirignano2018,Berg2018,Zang2019,Li2019,Beck2019,Cai2019,Raissi2019}). Basically, the unknown solution is approximated by a DNN, which is trained through the minimization of specific loss functions. Thanks to the special nonlinear structure, the number of free parameters in a DNN increases mildly with the dimension, making high-dimensional approximations feasible with limited storage. On the other hand, the approximation errors of neural networks have dimension-independent rates for special functions (e.g., continuous functions \cite{Shen2021_3}, H\"older functions \cite{Shen2021_1,Shen2021_2,Jiao2021} and Barron functions \cite{Barron1992,Barron1993,Klusowski2018,E2019,E2020_2,Siegel2020_1,Siegel2020_2,Caragea2020}), hence overcoming or lessening the curse of dimensionality.

\subsection{Problem description}
We focus on the development of DNN-based methods for two typical categories of evolution equations, the parabolic and hyperbolic equations. As an example, we first consider the linear parabolic equation. Given a bounded domain $\Omega\subset\mathbb{R}^d$ with $d\in\mathbb{N}^+$ and $T>0$, we aim to solve
\begin{equation}\label{01}
\begin{cases}
\frac{\partial u(x,t)}{\partial t}+\mathcal{L}_xu(x,t)=f(x,t),\quad\text{in}~\Omega\times(0,T],\\
u(x,0) = g(x),\quad\text{in}~\Omega,\\
u(x,t) = h(x,t),\quad\text{on}~\partial\Omega\times[0,T],
\end{cases}
\end{equation}
where $\mathcal{L}_x$ is a uniformly elliptic operator only involving $x$; $f$, $g$, and $h$ are given functions. Several methods such as physics-informed neural networks \cite{Raissi2019} and deep Galerkin method \cite{Sirignano2018} are developed for similar problems. Overall, these methods can be generalized as the DNN-based least-squares method. More precisely, let $\mathcal{N}$ be a generic class of DNNs, it searches for the solution of \eqref{01} in $\mathcal{N}$ by the following minimization problem,
\begin{gather}\label{03}
\min_{\hu\in\mathcal{N}}~J[\hu]:=\left\|\frac{\partial \hu}{\partial t}+\mathcal{L}_x\hu-f\right\|_{L^2(\Omega\times(0,T])}^2+\lambda_1\left\|\hu(\cdot,0)-g\right\|_{L^2(\Omega)}^2+\lambda_2\left\|\hu-h\right\|_{L^2(\partial\Omega\times[0,T])}^2,\\
\text{with~}\lambda_1,~\lambda_2>0,\notag
\end{gather}
where the loss function is defined as the square sum of equation error and condition errors, and the minimizer $\hu$ is taken as an approximate solution. Note that if \eqref{01} is a well-posed problem, then $\|u\|$ is bounded by $O(\|f\|+\|g\|+\|h\|)$. By the linearity, the error $\|\hu-u\|$ is bounded by the loss function $J[\hu]$. It implies that if a network $\hu$ has a small value of $J[\hu]$, it is close to the true solution $u$.

In practice, the $L^2$ integrals in $J$ are usually discretized and computed by Monte Carlo method (especially when $d$ is large), and the minimization can be solved under systematic deep learning frameworks. The deep least-squares method enjoys the advantages that it is friendly in high dimensions and flexible for various kinds of equations and domains. However, the main difficulty lies in solving the optimization \eqref{03} numerically. Due to the high non-convexity of the loss function, usual optimizers (e.g., stochastic gradient descent) might converge to bad local minimizers if the initialization or hyper-parameters are not configured well. For this issue, one approach is to simplify the loss function. In \cite{Gu2021}, Gu et al. constructed special DNN structures according to the initial/boundary conditions. More precisely, for the problem \eqref{01}, one can design a DNN-based class $\mathcal{N}'$ such that
\begin{equation}
\hu(x,0) = g(x)~\text{in}~\Omega,\quad \hu(x,t) = h(x,t)~\text{on}~\partial\Omega\times[0,T],\quad\forall \hu\in\mathcal{N}'.
\end{equation}
With hypothesis space $\mathcal{N}'$, the condition errors in $J$ automatically vanish. Then the optimization \eqref{03} can be simplified as follows
\begin{equation}\label{44}
\min_{\hu\in\mathcal{N}'}~J'[\hu]:=\left\|\frac{\partial \hu}{\partial t}+\mathcal{L}_x\hu-f\right\|_{L^2(\Omega\times(0,T])}^2.
\end{equation}

This design leads to an easier implementation and facilitates the convergence of optimizers to some extent. Nevertheless, a number of experiments show if the true solution is fairly oscillatory, then optimizers are still likely to find bad local minimizers. For example, the least-squares method fails in the 2-D heat equation with a solution having eight trigonometric waves in time (see Section \ref{sec_case1}). The drawback of least-squares method inspires us to look for more accurate deep methods for underlying oscillatory solutions.

\subsection{Contribution}
In evolution equations, more caution is required for the discretization of the time variable than the space variable, because the stability of the solution is usually sensitive to the approximation scheme in time. Therefore, different from the deep least-squares method that deals with time and space simultaneously in (\ref{03}), we consider addressing the
numerical solution in time and space separately by two distinct techniques respectively. On one hand, the time varies merely in one dimension, which allows us to employ
traditional numerical techniques for the time discretization. Among all available techniques, the spectral-Galerkin method \cite{Shen2011} is well-known for its high accuracy for high-regularity solutions. Especially, if the true solution is $C^\infty$, then the approximate solution obtained by the spectral-Galerkin method has an exponential decaying error order in terms of the degree of freedom, which is much more accurate than other techniques such as linear multistep methods and Runge-Kutta methods. On the other hand, the space variable (possibly of high dimensions) that is intractable by traditional techniques can be handled appropriately
using DNN-based techniques. Consequently, it is promising to develop a hybrid approach combining the traditional high-accuracy Galerkin method and the recent dimension-friendly deep methods.

In this work, we put forward a novel deep adaptive basis Galerkin (DABG) method to solve parabolic and hyperbolic equations. Instead of using a single DNN to approximate the solution as in the previous work, we propose a special adaptive basis structure for approximation. Specifically, the structure is formulated as a tensor product of a sequence of fixed 1-D orthogonal polynomials in time and a sequence of unknown DNNs in space.
According to this structure, two computational steps are performed to solve the PDE. First, we apply the Galerkin method to the original equation to eliminate the time variable, leading to a linear system of differential equations having unknown DNNs. Thanks to the orthogonality of the polynomials, the system is sparse with unknown DNNs. Next, we solve the resulting DNN system by minimizing its residual under a deep least-squares framework.

The theoretical analysis is also conducted for the proposed numerical method. We first develop a posterior estimates that implies the solution error is bounded by the minimal loss function and the term $O(N^{-m})$, where $N$ is the number of basis functions and $m$
is used to characterize the regularity of right hand side function $f$. Moreover, given that the true solution is a Barron-type function, we derive a priori estimates quantifying the error bound by $M$ and $N$, where $M$ is the width of the used networks. We show that
the solution converges to zero as long as $M=O(N^p)$ and $N$ tends to infinity
for some constant power $p$.

In the numerical experiments, we implement DABG method on a series of problems to validate its performance. First, for the low-dimensional heat equation, DABG obtains numerical solutions with errors as small as $O(10^{-8})$, which is much more accurate than the compared deep least-squares method. To the best of our knowledge, in the experiments of recently published literature, DNN-based methods can only achieve the best error $O(10^{-4})$. Furthermore, for the high-dimensional problems with oscillatory solutions, DABG can still preserve high accuracy while deep least-squares method
performs quite worse. Besides, an example of high-dimensional Allen-Cahn equation is
considered and solved by DABG, implying the proposed method can be generalized for nonlinear equations. The performance of the proposed method is quite well.

\subsection{Organization}
This paper is organized as follows. In Section 2, we review Galerkin method with orthogonal polynomial basis for the first and second-order ordinary differential equations. In Section 3, we introduce the adaptive basis structure, followed by the DABG method for parabolic equations is presented. We also conduct error estimates for the proposed method. In Section 4, we derive the similar DABG method for hyperbolic equations as well as the corresponding error estimates. Several numerical examples are demonstrated in Section 5 to test the performance of the proposed method. Conclusions and a discussion about further research works are provided in Section 6.

\section{Preliminaries}
In this section, we mainly introduce the Galerkin method for first and second-order ordinary differential equations, which provides a foundation for the DABG method discussed in following sections. We specifically use orthogonal polynomials to construct the basis functions because of the high accuracy of the approximation and the sparsity of the derived linear system. Such method is also referred as spectral-Galerkin method. We remark that other types of basis functions can be also applied similarly, such as the finite element method \cite{Ern2004}. Throughout this paper, we use $(\cdot,\cdot)$ to denote the $L^2$-inner product over $[0,T]$, namely, $(p,q)=\int_0^Tp(t)q(t)\td t$.

\subsection{Orthogonal polynomials}
We first introduce the orthogonal polynomials, which will be adopted to construct the basis functions in the following Galerkin method. Let $I:=[a,b]\subset\mathbb{R}$ be an interval and $N$ be a positive integer, we use $P^N(I)$ to denote the space of all polynomials in $I$ of the degree no greater than $N$. We also define
\begin{equation}
P^N_{0^-}(I):=\left\{p\in P^N(I),~p(a)=0\right\},\quad P^N_{0^+}(I):=\left\{p\in P^N(I),~p(b)=0\right\}
\end{equation}
as subspaces of $P^N(I)$.

In the case $I=[-1,1]$, Legendre polynomials, which are given by
\begin{equation}
L_0(z)=1,\quad L_1(z)=z,\quad L_{n+1}(z)=\frac{2n+1}{n+1}zL_n(z)-nL_{n-1}(z)~\text{for}~1\leq n
3\leq N-1,
\end{equation}
form an orthogonal basis for $P^N([-1,1])$ in the sense that
\begin{equation}\label{09}
\int_{-1}^1L_n(z)L_m(z)=\frac{2}{2n+1}\delta_{mn},\quad\forall m,n\geq0.
\end{equation}
Also, we have the boundary values
\begin{equation}\label{07}
L_n(\pm1)=(\pm1)^n,\quad L_n'(\pm1)=\frac{1}{2}(\pm1)^{n-1}n(n+1),\quad\forall n\geq0.
\end{equation}

\subsection{Spectral-Galerkin methods for 1-D problems}\label{sec_method_1D}
For completeness, we present the spectral-Galerkin methods by solving the following 1-D model problems:
\begin{itemize}
  \item first-order equation\begin{equation}\label{04}\begin{cases}u'(t)+\alpha u(t)=f(t),\quad t\in(0,T],\\u(0)=0;\end{cases}\end{equation}
  \item second-order equation\begin{equation}\label{05}\begin{cases}u''(t)+\alpha u(t)=f(t),\quad t\in(0,T],\\u(0)=0,~u'(0)=g_0,\end{cases}\end{equation}
\end{itemize}
where $\alpha\in\mathbb{R}$ is a given constant.

\subsubsection{First-order equation}\label{sec_first_order_equation}
Firstly, let us consider the first-order equation \eqref{04}. Due to the initial condition $u(0)=0$, we let $X^{(1)}:=\{u\in H^1([0,T]),~u(0)=0\}$ be the trial space, where $H^1([0,T])$ is the first-order Sobolev space in $[0,T]$. Also, we let the space $Y^{(1)}:=L^2([0,T])$ be the test space. Multiplying $v\in Y$ on both sides of \eqref{04} leads to the variational formulation,
\begin{equation}
\begin{cases}
\text{Find}~u\in X^{(1)}~\text{such that}\\
(u',v)+\alpha(u,v)=(f,v),~\forall v\in Y^{(1)}.
\end{cases}
\end{equation}
In the sense of approximation, we let $X_N^{(1)}:=P^N_{0^-}([0,T])\subset X^{(1)}$ be the trial space and $Y_N^{(1)}:=P^{N-1}([0,T])\subset Y^{(1)}$ be the test space (note that $\dim(X_N^{(1)})=\dim(Y_N^{(1)})$), then the Galerkin formulation is given by
 \begin{equation}\label{06}
\begin{cases}
\text{Find}~u_N\in X_N^{(1)}~\text{such that}\\
(u_N',v_N)+\alpha(u_N,v_N)=(f,v_N),~\forall v_N\in Y_N^{(1)}.
\end{cases}
\end{equation}
Note that in \eqref{06} the test space is not the same as the trial space, it is actually a Petrov-Galerkin formulation. Although the formulation \eqref{06} can be solved by various methods such as finite element method, we proposed the spectral method with orthogonal polynomial-type basis which has higher accuracy for smooth solutions. Specifically, define
\begin{equation}\label{55}
\phi_n^{(1)}(t):=L_n(\frac{2t}{T}-1)+L_{n-1}(\frac{2t}{T}-1),\quad\forall n\geq1.
\end{equation}
Using the property \eqref{07}, it can be verified that $\{\phi_n^{(1)}\}_{n=1}^N$ form a basis for $X_N^{(1)}$. On the other hand, define
\begin{equation}
~\psi_n^{(1)}(t):=\begin{cases}L_0(\frac{2t}{T}-1),\quad n=1,\\L_{n-1}(\frac{2t}{T}-1)-L_{n-2}(\frac{2t}{T}-1),\quad n\geq2,\end{cases}
\end{equation}
then it can be verified that $Y_N^{(1)}=\text{span}\{1\}\oplus P^{N-1}_{0^+}([0,T])=\text{span}\{\psi_1^{(1)}\}\oplus\text{span}\{\psi_n^{(1)}\}_{n=2}^N$, so $\{\psi_n^{(1)}\}_{n=1}^N$ forms a basis for $Y_N^{(1)}$. The special design of basis functions contributes to the sparsity of the derived linear system (see Proposition \ref{prop01}).

Next, we let
\begin{equation}\label{10}
u_N(t)=\sum_{n=1}^N\tdu_n\phi_n^{(1)}(t)\in X_N^{(1)}
\end{equation}
be the approximate solution with undetermined coefficients $\tdu_n$. Taking $v_N=\psi_j^{(1)}$ for $j=1,\cdots,N$ in \eqref{06} leads to the following linear system
\begin{equation}\label{08}
\sum_{n=1}^N\left(a_{jn}^{(1)}+\alpha b_{jn}^{(1)}\right)\tdu_n=(f,\psi_j^{(1)}),\quad j=1,\cdots,N,
\end{equation}
where
\begin{equation}\label{18}
a_{jn}^{(1)}:=((\phi_n^{(1)})',\psi_j^{(1)}),\quad b_{jn}^{(1)}:=(\phi_n^{(1)},\psi_j^{(1)}).
\end{equation}
The matrices in \eqref{08} have band structures due to the following result.

\begin{prop}\label{prop01}
It satisfies
\begin{equation}
a_{jn}^{(1)}=\begin{cases}2, \ \ \quad j=1,\\-2,\quad j=n+1\geq2,\\0, \ \ \quad\text{\rm otherwise},\end{cases}\quad
b_{jn}^{(1)}=\begin{cases}\frac{T}{2n-1},\quad \quad \quad \quad \quad j=n,\\-\frac{2T}{(2n-1)(2n+1)},\quad j=n+1,\\-\frac{T}{2n+1},
\quad \quad \quad \quad j=n+2,\\0,\quad \quad \quad \quad \quad \quad\text{\rm otherwise}.\end{cases}
\end{equation}
\end{prop}
\begin{proof}
From the orthogonality \eqref{09}, note that $\int_{-1}^1L_n(z)p(z)\td z=0$ for all $p\in P^m([-1,1])$ if $m<n$. Let us first consider the integral $a_{jn}^{(1)}=\int_0^T(\phi_n^{(1)})'\psi_j^{(1)}\td t$. The property \eqref{07} and integration by parts will be used repeatedly.

In the case $j\geq2$, if $n<j-1$, by the transformation $t=T(z+1)/2$,
\begin{equation}
\int_0^T(\phi_n^{(1)})'\psi_j^{(1)}\td t=\int_{-1}^1(L_n'(z)+L_{n-1}'(z))(L_{j-1}(z)-L_{j-2}(z))\td z=0
\end{equation}
since $L_n'(z)+L_{n-1}'(z)\in P^{n-1}([-1,1])$. Similarly, if $n=j-1$,
\begin{multline}
\int_0^T(\phi_{j-1}^{(1)})'\psi_n^{(1)}\td t=\int_{-1}^1(L_{j-1}'(z)+L_{j-2}'(z))(L_{j-1}(z)-L_{j-2}(z))\td z\\
=-\int_{-1}^1L_{j-1}'(z)L_{j-2}(z)\td z=\int_{-1}^1L_{j-1}(z)L_{j-2}'(z)\td z-L_{j-1}(1)L_{j-2}(1)+L_{j-1}(-1)L_{j-2}(-1)=-2;
\end{multline}
if $n>j-1$,
\begin{multline}
\int_0^T(\phi_n^{(1)})'\psi_j^{(1)}\td t=-\int_0^T\phi_n^{(1)}(\psi_j^{(1)})'\td t+\phi_n^{(1)}(T)\psi_j^{(1)}(T)-\phi_n^{(1)}(0)\psi_j^{(1)}(0)\\
=-\int_0^T\phi_n^{(1)}(\psi_j^{(1)})'\td t=\int_{-1}^1(L_{n-1}(z)+L_n(z))(L_{j-1}'(z)-L_j'(z))\td z=0.
\end{multline}

In the case $j=1$,
\begin{equation}
\int_0^T(\phi_n^{(1)})'\psi_1^{(1)}\td t=\int_{-1}^1L_n'(z)+L_{n-1}'(z)\td z=L_n(1)+L_{n-1}(1)-L_n(-1)-L_{n-1}(-1)=2.
\end{equation}

The result for $b_{jn}^{(1)}$ can be proved by similar argument.
\end{proof}

Proposition \ref{prop01} implies that the matrices
\begin{equation}
A_N^{(1)}:=\left[a_{jn}^{(1)}\right]_{j=1,\cdots,N}^{n=1,\cdots,N}~\text{and}~B_N^{(1)}:=\left[b_{jn}^{(1)}\right]_{j=1,\cdots,N}^{n=1,\cdots,N}
\end{equation}
are sparse with $O(N)$ nonzero entries. The band structures of the two matrices are shown in Figure \ref{Fig_matrix_1}.

\begin{figure}
\centering
\includegraphics[scale=0.45]{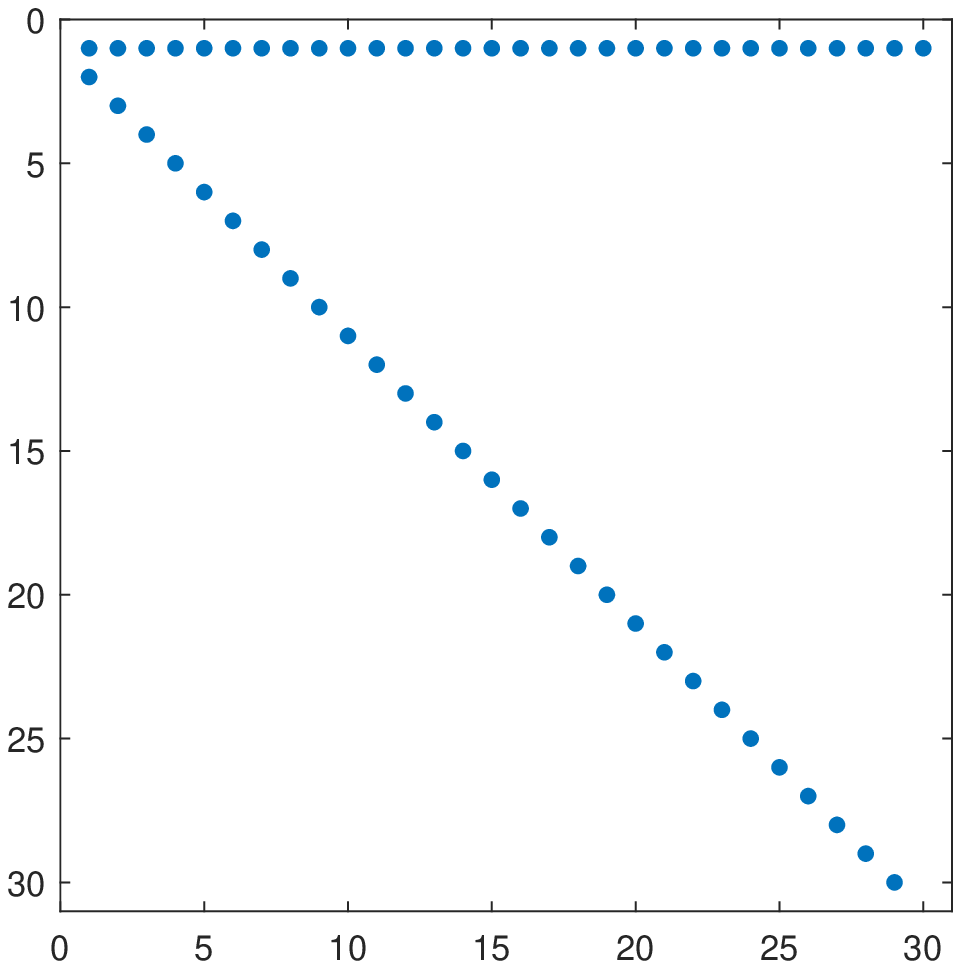}
\includegraphics[scale=0.45]{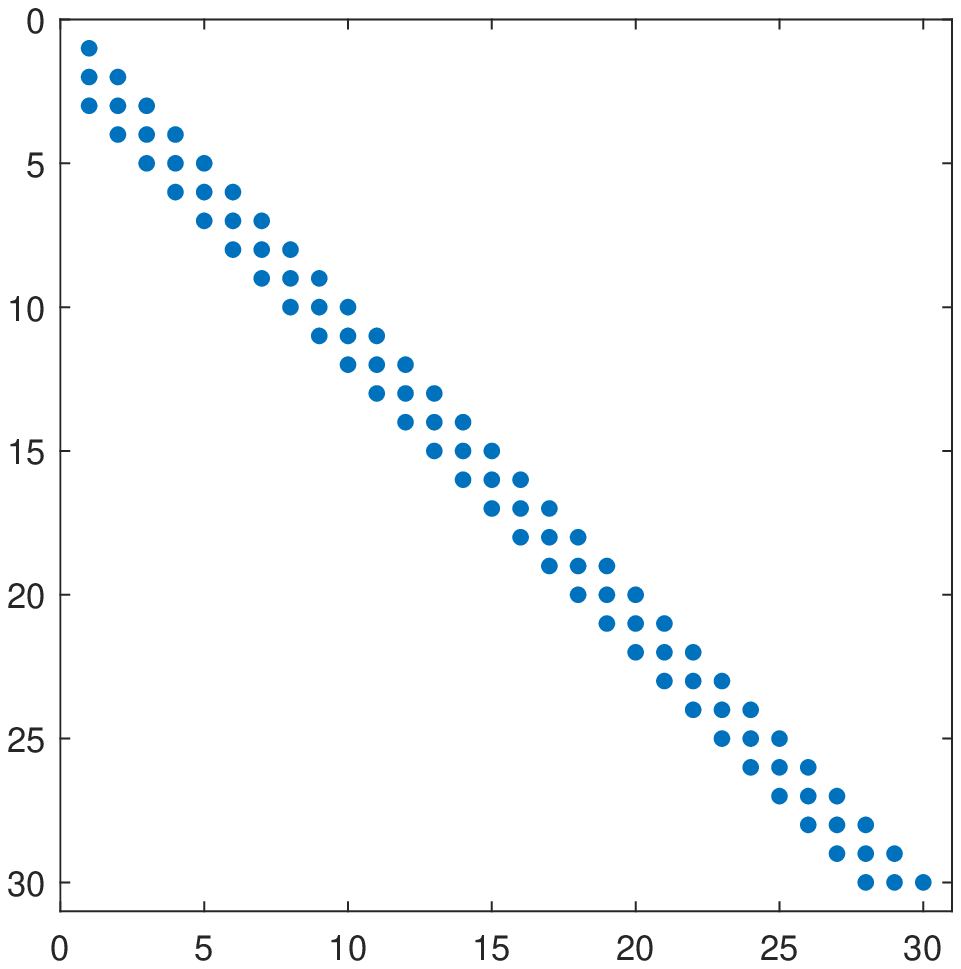}
\caption{\em Band structures of the matrices $A_N^{(1)}$ (left) and $B_N^{(1)}$ (right) when $N=30$.}
\label{Fig_matrix_1}
\end{figure}

\subsubsection{Second-order equation}\label{sec_second_order_equation}
Next, we briefly present the variational and Galerkin formulations for the second-order equation \eqref{05} derived by similar argument as above. Define the trial and test spaces
\begin{equation}
X^{(2)}:=\{u\in H^1([0,T]),~u(0)=0\},\quad Y^{(2)}:=\{v\in H^1([0,T]),~v(T)=0\},
\end{equation}
and the approximate spaces
\begin{equation}
X_N^{(2)}:=P^N_{0^-}([0,T])\subset X^{(2)},\quad Y_N^{(2)}:=P^N_{0^+}([0,T])\subset Y^{(2)}.
\end{equation}

It can be noted that $X^{(2)}$ and $X_N^{(2)}$ are identical to $X^{(1)}$ and $X_N^{(1)}$, respectively, although we use the label ``(1)" and ``(2)" to distinguish them in the orders of the equation. However, different from $Y^{(1)}=L^2([0,T])$ in the first-order case, $Y^{(2)}$ is taken as the $H^1$ space with homogeneous Dirichlet condition on the right end point (same for $Y_N^{(2)}$).

We remark that even though the essential condition $u(0)=0$ is enforced in the trial space, the natural condition $u'(0)=g_0$ is not enforced explicitly and will be absorbed in the bilinear form. By multiplying $v\in Y^{(2)}$ on both sides of \eqref{05} and using integration by parts, the variational formulation for \eqref{05} is can be derived as
\begin{equation}
\begin{cases}
\text{Find}~u\in X^{(2)}~\text{such that}\\
(u',v')-\alpha (u,v)=-(f,v)-g_0v(0),~\forall v\in Y^{(2)},
\end{cases}
\end{equation}
and its Galerkin formulation is given by
 \begin{equation}\label{13}
\begin{cases}
\text{Find}~u_N\in X_N^{(2)}~\text{such that}\\
(u_N',v_N')-\alpha(u_N,v_N)=-(f,v_N)-g_0v_N(0),~\forall v_N\in Y_N^{(2)}.
\end{cases}
\end{equation}

Moreover, define
\begin{equation}
\phi_n^{(2)}(t):=\begin{cases}L_{1}(\frac{2t}{T}-1)+L_{0}(\frac{2t}{T}-1),\quad n=1,\\(1-\frac{1}{n})L_{n}(\frac{2t}{T}-1)+(2-\frac{1}{n})L_{n-1}(\frac{2t}{T}-1)+L_{n-2}(\frac{2t}{T}-1),\quad n\geq2,\end{cases}
\end{equation}
and
\begin{equation}\label{45}
\psi_n^{(2)}(t):=\begin{cases}L_{1}(\frac{2t}{T}-1)-L_{0}(\frac{2t}{T}-1),\quad n=1,\\(1-\frac{1}{n})L_{n}(\frac{2t}{T}-1)-(2-\frac{1}{n})L_{n-1}(\frac{2t}{T}-1)+L_{n-2}(\frac{2t}{T}-1),\quad n\geq2.\end{cases}
\end{equation}
Using the property \eqref{07}, it can be verified that $\{\phi_n^{(2)}\}_{n=1}^N$ and $\{\psi_n^{(2)}\}_{n=1}^N$ form a basis for $X_N^{(2)}$ and $Y_N^{(2)}$, respectively.

Finally, the approximate solution $u_N=\sum_{n=1}^N\tdu_n\phi_n^{(2)}$ can be determined by solving
\begin{equation}\label{12}
\sum_{n=1}^N\left(a_{jn}^{(2)}-\alpha b_{jn}^{(2)}\right)\tdu_n=-(f,\psi_j^{(2)})-g_0\psi_j^{(2)}(0),\quad j=1,\cdots,N,
\end{equation}
where
\begin{equation}\label{20}
a_{jn}^{(2)}:=((\phi_n^{(2)})',(\psi_j^{(2)})'),\quad b_{jn}^{(2)}:=(\phi_n^{(2)},\psi_j^{(2)}).
\end{equation}

By the fact that $\partial_t\phi_n^{(2)}(0)=\partial_t\psi_n^{(2)}(T)=0$ if $n\geq2$ and a similar proof as in Proposition \ref{prop01}, we can deduce the following result which implies that the matrix in \eqref{12} has band structures.
\begin{prop}\label{prop02}
It satisfies
\begin{equation}
a_{jn}^{(2)}=
\left \{
\begin{array}{ll}
\frac{4(2n-1)}{Tn}, & \ j=1,\\
-\frac{4(2n-1)}{Tn}, & \ j\geq2,n=1,\\
-\frac{4(n-1)(2n-1)}{Tn}, & \ j=n\geq2,\\
0, & \ \text{\rm otherwise},
\end{array}
\right.
\quad
b_{jn}^{(2)}=
\left \{
\begin{array}{ll}
-\frac{2}{3}, & \ j=n=1, \\
-\frac{2T(2n-1)(n^2-n-3)}{n^2(2n-3)(2n+1)}, & \ j=n\geq2,\\
\frac{1}{2}, & \ j=n+1=2, \\
\frac{2T}{n(n+1)}, & \ j=n+1\geq3,\\
\frac{1}{3}, & \ j=n+2=3, \\
\frac{T(n-1)}{n(2n+1)}, & \ j=n+2\geq4,\\
-b_{nj}^{(2)}, & \ j=n-1,\\
b_{nj}^{(2)}, & \ j=n-2,\\
0, & \ \text{\rm otherwise}.
\end{array}
\right.
\end{equation}
\end{prop}

Clearly, the matrices
\begin{equation}
A_N^{(2)}:=\left[a_{jn}^{(2)}\right]_{j=1,\cdots,N}^{n=1,\cdots,N}~\text{and}~B_N^{(2)}:=\left[b_{jn}^{(2)}\right]_{j=1,\cdots,N}^{n=1,\cdots,N}
\end{equation}
are sparse with $O(N)$ nonzero entries (see Figure \ref{Fig_matrix_2}).

\begin{figure}
\centering
\includegraphics[scale=0.45]{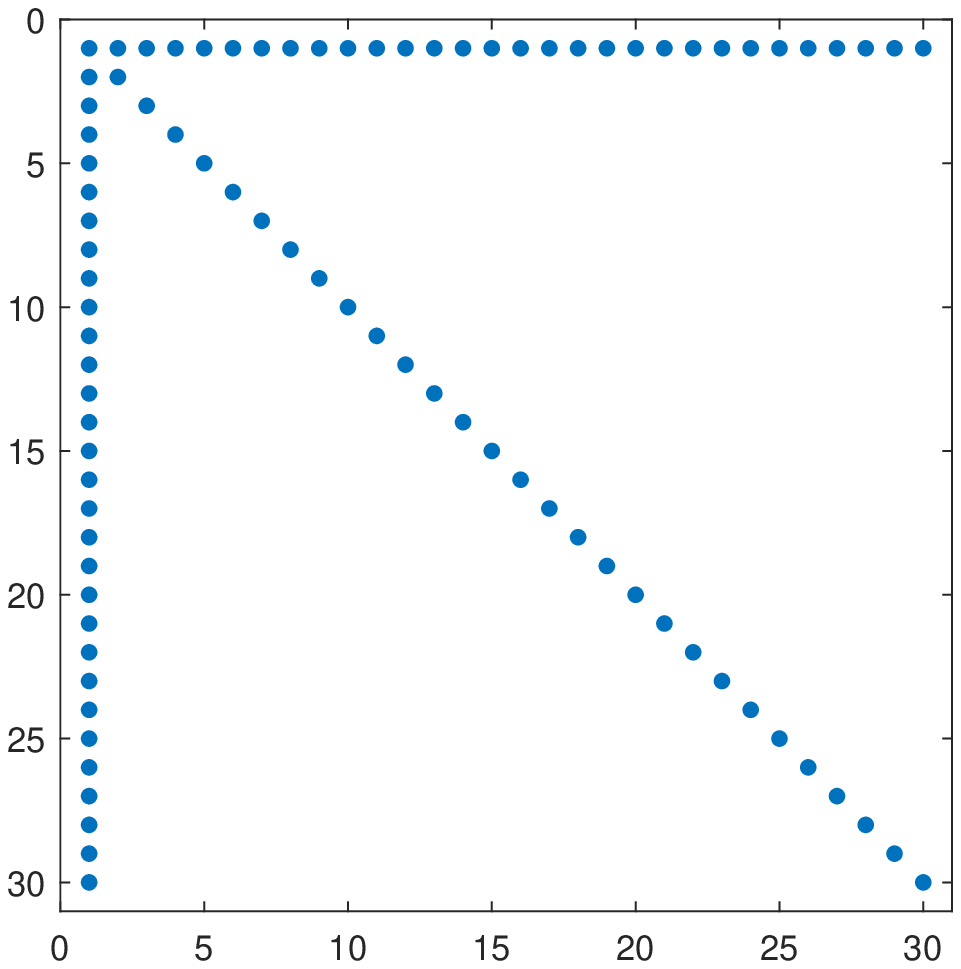}
\includegraphics[scale=0.45]{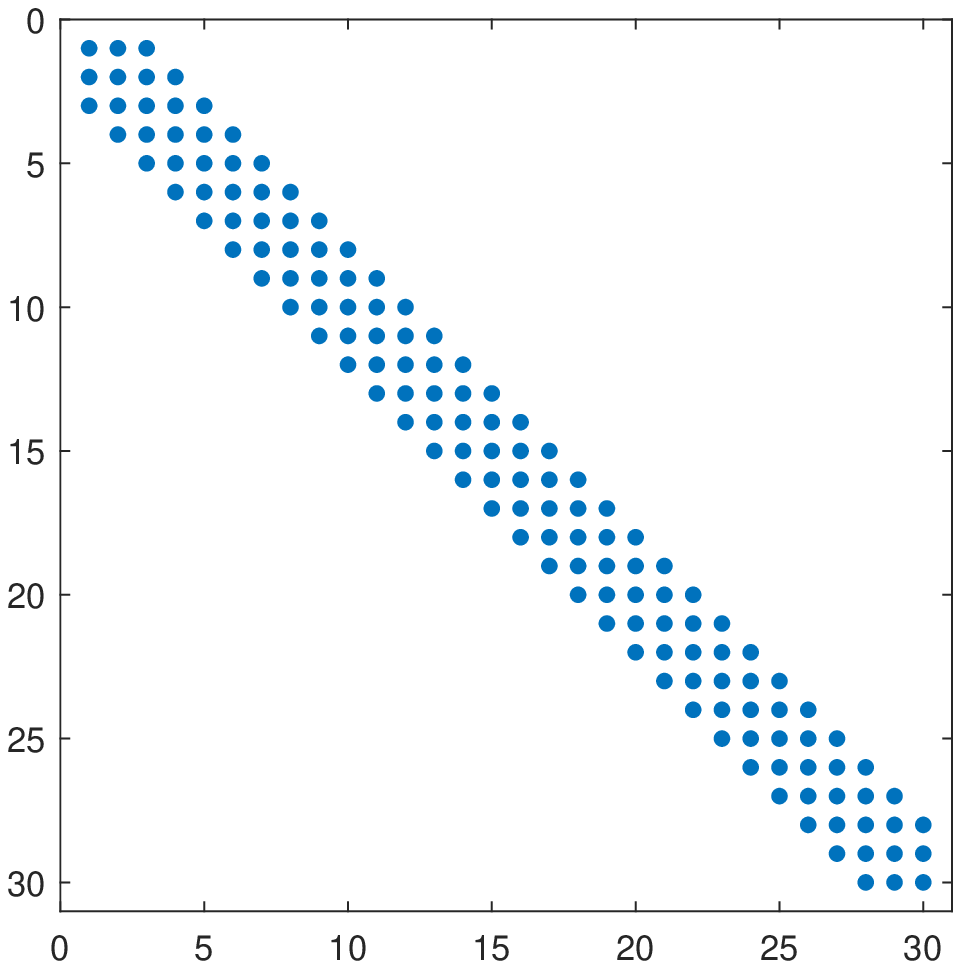}
\caption{\em Band structures of the matrices $A_N^{(2)}$ (left) and $B_N^{(2)}$ (right) when $N=30$.}
\label{Fig_matrix_2}
\end{figure}

\section{Parabolic equations}\label{sec_method_parabolic}
In section, we present the deep adaptive basis Galerkin method for linear parabolic equation \eqref{01} in detail. For simplicity, the inhomogeneous initial and boundary conditions in \eqref{01} can be converted to homogeneous ones by lifting techniques. Specifically, we can construct a smooth function $v(x,t)$ such that $v(x,0)=g(x)$ in $\Omega$ and $v=h$ on $\partial\Omega\times(0,T)$, then the function $w:=u-v$ satisfies $\frac{\partial w}{\partial t}+\mathcal{L}_xw=f-\frac{\partial v}{\partial t}-\mathcal{L}_xv$ in $\Omega\times(0,T)$ with homogeneous conditions. Therefore, without loss of generality, we consider the following equation,
\begin{equation}\label{02}
\begin{cases}
\frac{\partial u(x,t)}{\partial t}+\mathcal{L}_xu(x,t)=f(x,t),\quad\text{in}~\Omega\times(0,T],\\
u(x,0) = 0,\quad\text{in}~\Omega,\\
u(x,t) = 0,\quad\text{on}~\partial\Omega\times[0,T].
\end{cases}
\end{equation}
For readers' convenience, we use the notation $\hat{\cdot}$ to denote functions which involve neural networks.

\subsection{Fully-connected neural network}
We first introduce the fully connected neural network (FNN), which is one of the most widely used neural networks and will be adopted to construct the following adaptive basis. Mathematically speaking, let $\sigma(x)$ be an activation function which is applied entry-wise to a vector $x$ to obtain another vector of the same size. Let $L\in\mathbb{N}^+$ and $M_\ell\in\mathbb{N}^+$ for $\ell=1,\dots,L-1$, an FNN $\hphi$ is the composition of $L-1$ simple nonlinear functions given by
\begin{equation}\label{36}
\hat{\phi}(x;\theta):=a^T h_{L-1} \circ h_{L-2} \circ \cdots \circ h_{1}(x)\quad \text{for } x\in\mathbb{R}^d,
\end{equation}
where $a\in \mathbb{R}^{M_{L-1}}$; $h_{\ell}(x_{\ell}):=\sigma\left(W_\ell x_{\ell} + b_\ell \right)$ with $W_\ell \in \mathbb{R}^{M_{\ell}\times M_{\ell-1}}$ ($M_0:=d$) and $b_\ell \in \mathbb{R}^{M_\ell}$ for $\ell=1,\dots,L-1$. Here $M_\ell$ is called the width of the $\ell$-th layer and $L$ is called the depth of the FNN. $\theta:=\{a,\,W_\ell,\,b_\ell:1\leq \ell\leq L-1\}$ is the set of parameters determining the underlying neural network $\hphi$. Popular activation functions in solving PDEs include the sigmoidal function $(1+e^{-x})^{-1}$ and $k$-th power of rectified linear unit (ReLU$^k$) $\max\{0,x^k\}$ for $k\in\mathbb{N}^+$.

Denote $M:=\max~\{M_\ell,~1\leq\ell\leq L-1\}$, then the number of free parameters $|\theta|=O(M^2L+Md)$. Comparatively, the degree of freedom of linear structures such as finite elements and tensor product polynomials increases exponentially with $d$. Hence FNNs are more practicable in high-dimensional approximations. Given an activation function, we consider the architecture with fixed width $M=M_\ell$ for all $\ell$ and denote $\mathcal{F}_{L,M}$ as the set consisting of all FNNs with depth $L$ and width $M$.

\subsection{Adaptive basis strucutres}
Recall for 1-D problems in Section \ref{sec_method_1D}, we use the expansion \eqref{10} for approximation. Based on that, for multi-dimensional parabolic equations, we replace the scalar coefficients $\tdu_n$ by space-dependent basis functions. Considering the spatial dimension $d$ might be moderately large, these basis can be constructed based on FNNs, thanks to their capability of high dimensions. Specifically, we aim to look for approximate solutions of the following form
\begin{equation}\label{11}
\hu_N(x,t)=\sum_{n=1}^N\hw_n(x)\phi_n^{(1)}(t),
\end{equation}
where $\hw_n:\mathbb{R}^d\rightarrow\mathbb{R}$ are functions belonging to a specified FNN-based class $\mathcal{S}$. Clearly, $\mathcal{S}$ should have regularity such that $\mathcal{L}_x\hw$ is meaningful for all $\hw\in\mathcal{S}$.

We expect the expansion \eqref{11} to satisfy the homogeneous boundary condition $\hu_N(x,t) = 0$ on $\partial\Omega\times(0,T)$. A convenient way is to enforce the boundary condition on $\hw_n$. That is, we expect that $\hw_n(x)=0$ on $\partial\Omega$ for all $n$. By this setting, the boundary condition is automatically satisfied by $\hu_N$, and we do not have to deal with the boundary condition explicitly in the formulation. Consequently, we require that
\begin{equation}\label{14}
\hw(x)=0~\text{on}~\partial\Omega,\quad\forall \hw\in\mathcal{S}.
\end{equation}

If we explicitly formulate $\mathcal{S}$, we need to ensure \eqref{14} is fulfilled. It usually depends on the specific shapes of $\Omega$. On the other hand, we need to ensure that the functions in $\mathcal{S}$ have FNN structures fundamentally. Therefore, we design $\mathcal{S}$ as follows,
\begin{equation}
\mathcal{S}=\left\{\hw: \hw(x)=\nu(x)\hphi(x),~\hphi\in\mathcal{F}_{L,M}\right\},
\end{equation}
where $\nu$ is a prescribed smooth function governing the boundary condition, namely, $\nu|_{\partial\Omega}=0$. Possible choices of $\mathcal{S}$ for various types of $\Omega$ are listed in Table \ref{Tab_S}. In Section \ref{sec_a_priori_estimate}, we will show if $\Omega$ is smooth, there is always a boundary governing function $\nu$ with the same smoothness (Lemma \ref{lem02}).
\begin{table}
\centering
\begin{tabular}{ccc}
\toprule[1pt]
  Domain Type & Formulation of $\Omega$ & $\mathcal{S}$ \\
\midrule
  $d$-box & $\prod_{i=1}^d(a_i,b_i)$ & $\{\hw:~\hw(x)=\prod_{i=1}^d(x_i-a_i)(x_i-b_i)\cdot\hphi,~\hphi\in\mathcal{F}_{L,M}\}$ \\
  $d$-ball & $\{x\in\mathbb{R}^d:~|x|<r\}$ & $\{\hw:~\hw(x)=(|x|^2-r^2)\hphi,~\hphi\in\mathcal{F}_{L,M}\}$ \\
  Interior of a level set & $\{x\in\mathbb{R}^d:~\rho(x)<a\}$ & $\{\hw:~\hw(x)=(\rho(x)-a)\hphi,~\hphi\in\mathcal{F}_{L,M}\}$ \\
  General Lipschitz & $\Omega\subset\mathbb{R}^d$,~$\partial\Omega$ is $C^{0,1}$ & $\{\hw:~\hw(x)=\text{dist}(x,\partial\Omega)\hphi,~\hphi\in\mathcal{F}_{L,M}\}$ \\
\bottomrule[1pt]
\end{tabular}
\caption{Possible choices of $\mathcal{S}$ for various types of the problem domain $\Omega$.}\label{Tab_S}
\end{table}

To sum up, we define the hypothesis space
\begin{equation}\label{15}
\mathcal{U}_N^\tP:=\left\{\hu_N(x,t)=\sum_{n=1}^N\hw_n(x)\phi_n^{(1)}(t),~\hw_n\in\mathcal{S}\right\},
\end{equation}
where the label ``P'' is short for the parabolic equation. We can formally regard $\mathcal{U}_N^\tP$ as a tensor-product class, i.e. $\mathcal{U}_N^\tP=\mathcal{S}\times X_N^{(1)}$. Since the ``basis functions" of the structure $\hu_N$ in \eqref{15} include neural networks, which are nonlinear structures with the free parameters, we can view $\hu_N$ as an adaptive basis structure. Moreover, all functions in $\mathcal{U}_N^\tP$ always satisfy the initial and boundary conditions in \eqref{02}. It suffices to look for approximate solutions in $\mathcal{U}_N^\tP$ that satisfy the differential equation.

We remark that even though we propose to take $N$ independent FNNs in the adaptive basis structure \eqref{15}, one can alternatively use a vector-valued FNN with $N$-dimensional output in the structure. More precisely, an $N$-dimensional vector-valued FNN $\hphi$ is defined as \eqref{36} by specifying $a\in\mathbb{R}^{M_{L-1}\times N}$, and the $n$-th component of $\hphi$ can be taken to build $\hw_n(x)$. It is clear that using one FNN with $N$-dimensional output is equivalent to using $N$ scalar FNNs sharing the same input/hidden layers $h_\ell$ but having different output layers $a$. Therefore, the hypothesis space constructed by one FNN with $N$-dimensional output is a subspace of the one with $N$ scalar FNNs. For a more general discussion, we choose to use the latter case in this paper.

\subsection{Galerkin formulation}
Following the method for the first-order problem in Section \ref{sec_first_order_equation}, we can formally derive the variational form for the parabolic problem \eqref{02}. Recall $X^{(1)}$, $Y^{(1)}$ and $Y_N^{(1)}$ are defined in Section \ref{sec_first_order_equation}, we define the trial space
\begin{equation}
\mathcal{U}^\tP:=\left\{u(x,t):u(x,\cdot)\in X^{(1)}~\forall x\in\Omega,~u|_{\partial\Omega\times[0,T]}=0\right\},
\end{equation}
then the variational form is given by
\begin{equation}
\begin{cases}
\text{Find}~u\in \mathcal{U}^\tP~\text{such that}\\
\left(\partial_tu(x,\cdot), v\right)+\left(\mathcal{L}_xu(x,\cdot), v\right)=\left(f(x,\cdot), v\right)~\forall v\in Y^{(1)},~\forall x\in\Omega.
\end{cases}
\end{equation}
Accordingly, the Galerkin formulation is given by
\begin{equation}\label{16}
\begin{cases}
\text{Find}~\hu_N\in \mathcal{U}_N^\tP~\text{such that}\\
\left(\partial_t\hu_N(x,\cdot), v_N\right)+\left(\mathcal{L}_x\hu_N(x,\cdot), v_N\right)=\left(f(x,\cdot), v_N\right)~\forall v_N\in Y_N^{(1)},~\forall x\in\Omega.
\end{cases}
\end{equation}
Taking $\hu_N$ as the expression \eqref{11} and $v_N=\psi_j^{(1)}$ for $j=1,\cdots,N$, we obtain
\begin{equation}\label{17}
\sum_{n=1}^N\left(a_{jn}^{(1)}\hw_n(x)+b_{jn}^{(1)}\mathcal{L}_x\hw_n(x)\right)=\left(f(x,\cdot),\psi_j^{(1)}\right),\quad j=1,\cdots,N,\quad\forall x\in\Omega,
\end{equation}
with unknowns $\{\hw_n\}\subset\mathcal{S}$, where $a_{jn}^{(1)}$ and $b_{jn}^{(1)}$ are defined in \eqref{18}. Note the equation in \eqref{17} holds for all $x\in\Omega$, we may not find a set of $\{\hw_n\}$ exactly making \eqref{17} satisfied. Therefore, in a practical way, we solve \eqref{17} by minimizing its residual in the least-squared sense,
\begin{equation}\label{22}
\underset{\hw_n\in\mathcal{S}}{\min}~J^{\tP}[\hw_1,\cdots,\hw_N]:=\sum_{j=1}^Nr_j\left\|R_j^{(1)}[\hw_1,\cdots,\hw_N]\right\|_{L^2(\Omega)}^2+\lambda N^{-4}\|\mathcal{L}_x\hw_N\|_{L^2(\Omega)}^2,
\end{equation}
where
\begin{equation}
R_j^{(1)}[\hw_1,\cdots,\hw_N]:=\sum_{n=1}^N\left(a_{jn}^{(1)}\hw_n(x)+b_{jn}^{(1)}\mathcal{L}_x\hw_n(x)\right)-\left(f(x,\cdot),\psi_j^{(1)}\right)
\end{equation}
is the residual of the $j$-th equation in \eqref{17};
\begin{equation}
r_j:=N^{-3}\sum_{k=j}^Nk(2k-1)
\end{equation}
between $O(N^{-1})$ and $O(1)$ is the weight of $R_j^{(1)}$; $\lambda>0$ is a given regularization parameter. Since the loss function is nonnegative, the optimization \eqref{22} always admits a solution. We remark that in \eqref{22} the weights $\{r_j\}$ and the regularization term $\lambda N^{-4}\|\mathcal{L}_x\hw_N\|_{L^2(\Omega)}^2$ are specifically chosen such that a decent a posterior estimate can be derived (see Section \ref{sec_a_posterior_estimate}). Also, by Proposition \ref{prop01}, the sum $\sum_{n=1}^N\left(a_{jn}^{(1)}\hw_n(x)+b_{jn}^{(1)}\mathcal{L}_x\hw_n(x)\right)$ in $R_j^{(1)}$ only includes at most 4 nonzero terms if $j\geq2$.

We remark that even though the regularization term $\lambda N^{-4}\|\mathcal{L}_x\hw_N\|_{L^2(\Omega)}^2$ is essential in the error analysis, it is less important in practice because it is much smaller than the residual part of the loss. More precisely, using \eqref{55} we obtain
\begin{equation}
\mathcal{L}_x\hu_N(x,t)=\mathcal{L}_x\hw_1(x)L_0(\frac{2t}{T}-1)+\sum_{n=1}^{N-1}(\mathcal{L}_x\hw_{n+1}(x)+\mathcal{L}_x\hw_n(x))L_n(\frac{2t}{T}-1)+\mathcal{L}_x\hw_N(x)L_N(\frac{2t}{T}-1);
\end{equation}
Namely, $\mathcal{L}_x\hw_N(x)$ is the coefficient associated with $L_N(\frac{2t}{T}-1)$ in the orthogonal expansion of $\mathcal{L}_x\hu_N(x,t)$ with the basis $\left\{L_n(\frac{2t}{T}-1)\right\}_{n=0}^N$. Due to the orthogonality, it holds that for each $x\in\Omega$,
\begin{equation}\label{61}
\mathcal{L}_x\hw_N(x)=\frac{\int_0^T\mathcal{L}_x\hu_N(x,t)\cdot L_N(\frac{2t}{T}-1)\td t}{\int_0^T\left|L_N(\frac{2t}{T}-1)\right|^2\td t}=\frac{2N+1}{T}\int_0^T\mathcal{L}_x\hu_N(x,t)\cdot L_N(\frac{2t}{T}-1)\td t.
\end{equation}
Using the identity of Legendre polynomials
\begin{equation}\label{62}
L_n(x)=(2n+1)^{-1}(L_{n+1}'(x)-L_{n-1}'(x))
\end{equation}
and integration by parts, we further obtain
\begin{multline}\label{60}
\int_0^T\mathcal{L}_x\hu_N(x,t)\cdot L_N(\frac{2t}{T}-1)\td t\\
=-\frac{T}{4N+2}\left(\int_0^T\mathcal{L}_x\partial_t\hu_N(x,t)L_{N+1}(\frac{2t}{T}-1)\td t - \int_0^T\mathcal{L}_x\partial_t\hu_N(x,t)L_{N-1}(\frac{2t}{T}-1)\td t\right).
\end{multline}
Since $|L_n(x)|\leq 1$ in $[-1,1]$ for all $n$, it follows \eqref{60}
\begin{multline}
\left|\int_0^T\mathcal{L}_x\hu_N(x,t)\cdot L_N(\frac{2t}{T}-1)\td t\right|\leq \frac{T}{2N+1}\left|\int_0^T\mathcal{L}_x\partial_t\hu_N(x,t)\td t\right|\\
\leq CN^{-1}\|\mathcal{L}_x\hu_N(x,\cdot)\|_{H^1([0,T])},
\end{multline}
for some constant $C$ only depending on $T$. Note that $\hu_N$ is infinitely differentiable in $t$, so we can do integration by parts on the resulting terms of \eqref{60} recursively for any number of times, and it finally leads to
\begin{equation}
|\mathcal{L}_x\hw_N(x)| = \frac{2N+1}{T}\left|\int_0^T\mathcal{L}_x\hu_N(x,t)\cdot L_N(\frac{2t}{T}-1)\td t\right|\leq CN^{1-p}\|\mathcal{L}_x\hu_N(\cdot,x)\|_{H^p([0,T])},
\end{equation}
where the integer $p$ can be arbitrarily large. This implies the regularization term $\lambda N^{-4}\|\mathcal{L}_x\hw_N\|_{L^2(\Omega)}^2$ decays exponentially to zero as $N\to\infty$. Therefore, if $N$ is moderately large, the regularization term can be neglected compared with the residual part of the loss that decays to zero in an algebraic rate as $M,N\to\infty$ (in the proof of Theorem \ref{thm02}).

\subsection{A posterior estimates}\label{sec_a_posterior_estimate}
We conduct error analysis for the proposed DABG formulation \eqref{22}. First, given an interval $I=[a,b]$, we introduce the weighted $L^2$-space with weight function $(b-z)^\alpha(z-a)^\alpha$, namely,
\begin{equation}
L^2_\alpha(I):=\left\{u:I\rightarrow\mathbb{R}~:\int_I|u|^2(b-z)^\alpha(z-a)^\alpha\td z<\infty\right\},\quad\alpha\in\mathbb{R},
\end{equation}
equipped with the norm
\begin{equation}
\|u\|_{L^2_\alpha(I)}:=\left(\int_I|u|^2(b-z)^\alpha(z-a)^\alpha\td z\right)^\frac{1}{2},
\end{equation}
and the non-uniformly Jacobi-weighted Sobolev space
\begin{equation}
\mathcal{W}^m(I):=\left\{u:I\rightarrow\mathbb{R}~:\partial_z^ku\in L^2_k(I),\quad0\leq k\leq m\right\},\quad m\in\mathbb{N},
\end{equation}
equipped with the norm
\begin{equation}
\|u\|_{\mathcal{W}^m(I)}:=\left(\sum_{k=0}^m\|\partial_z^ku\|_{L^2_k(I)}^2\right)^\frac{1}{2}.
\end{equation}
Also, we define the $L^2$-orthogonal projection $\pi^N_I:~L^2(I)\rightarrow P^N(I)$ such that
\begin{equation}
\int_I(\pi^N_Iu-u)v\td z=0,\quad\forall v\in P^N(I),
\end{equation}
given $u\in L^2(I)$. In the case $I=[-1,1]$, note that the Legendre polynomials $\{L_n\}_{n=0}^\infty$ forms an orthogonal basis for $L^2([-1,1])$, if we expand
$$
u(z)=\sum_{n=0}^\infty\tdu_nL_n(z),
$$
then
$$
\pi^N_{[-1,1]}u(z)=\sum_{n=0}^N\tdu_nL_n(z).
$$
The following lemmas show the approximation error between $u$ and $\pi^N_{[-1,1]}u$.
\begin{lemma}[Theorem 3.35, \cite{Shen2011}]
For any $u\in \mathcal{W}^m([-1,1])$, if $N\gg m$, we have
\begin{equation}
\|\pi^N_{[-1,1]}u-u\|_{L^2([-1,1])}\leq cN^{-m}\|\partial_z^mu\|_{L_m^2([-1,1])},
\end{equation}
where $c$ is a constant number.
\end{lemma}

By rescaling the interval, we directly have
\begin{corollary}\label{cor01}
For any $u\in \mathcal{W}^m([0,T])$, if $N\gg m$, we have
\begin{equation}
\|\pi^N_{[0,T]}u-u\|_{L^2([0,T])}\leq cN^{-m}\|\partial_t^mu\|_{L_m^2([0,T])},
\end{equation}
where $c$ is a constant number only depending on $m$ and $T$.
\end{corollary}

In the following passage, we will apply $\pi^N_{[0,T]}$ to functions of variables $x$ and $t$, which means we apply the orthogonal projection $\pi^N_{[0,T]}$ with respect to $t$ for each $x$ pointwise.

Now we derive the a posterior estimate for the DABG method. For simplicity, we only consider the heat equation with $\mathcal{L}_x=-\Delta_x$. It can be derived that the error is bounded by the final minimized loss $J^{\tP}$ and $O(N^{-m})$, if $f$ has $\mathcal{W}^m([0,T])$ regularity in $t$. The analysis can be easily generalized for other types of parabolic equations.
\begin{theorem}\label{thm01}
Let $u$ be the classical solution of \eqref{02} with $\mathcal{L}_x=-\Delta_x$. Let $m\geq0$ be some integer. Suppose that $f\in\mathcal{W}^m(0,T;L^2(\Omega))$, $u\in L^2(0,T;H_0^1(\Omega))$ and $\partial_tu\in L^2(0,T;H^{-1}(\Omega))$. Let $\hu_N(x,t)=\sum_{n=1}^N\hw_n(x)\phi_n^{(1)}(t)$, where $\{\hw_n\}_{n=1}^N$ solves the minimization \eqref{22}. Suppose that $\hw_n\in H_0^1(\Omega)\cap H^{-1}(\Omega)$ is twice differentiable for $n=1,\cdots,N$. Denote $\delta=J^{\tP}[\hw_1,\cdots,\hw_N]$ as the minimized loss. Then
\begin{multline}
\underset{0\leq t\leq T}{\sup}\|u(\cdot,t)-\hu_N(\cdot,t)\|_{H^1_0(\Omega)}+\|u-\hu_N\|_{L^2(0,T;H^2(\Omega))}+\|\partial_t(u-\hu_N)\|_{L^2(0,T;L^2(\Omega))}\\
\leq c\left(N^\frac{3}{2}\sqrt{\delta}+N^{-m} \|\partial_t^m f\|_{L_m^2(0,T;L^2(\Omega))}\right)
\end{multline}
provided that $N\gg m$, where $c$ is a constant only depending on $\Omega$, $m$, $T$ and $\lambda$.
\end{theorem}

\begin{proof}
Let $\he_N:=\hu_N-u$, then it satisfies
\begin{equation}
\begin{cases}
\partial_t \he_N(x,t)-\Delta_x\he_N(x,t)=\hp(x,t):=\partial_t \hu_N(x,t)-\Delta_x\hu_N(x,t)-f(x,t),\quad\text{in}~\Omega\times(0,T],\\
\he_N(x,0) = 0,\quad\text{in}~\Omega,\\
\he_N(x,t) = 0,\quad\text{on}~\partial\Omega\times[0,T].
\end{cases}
\end{equation}
By the hypothesis, $\hp\in L^2(0,T;L^2(\Omega))$, $\he_N\in L^2(0,T;H_0^1(\Omega))$ and $\partial_t\he_N\in L^2(0,T;H^{-1}(\Omega))$. Then using a priori estimates for heat equations (e.g., Theorem 5, Page. 382 in \cite{Evans2010}), we can obtain
\begin{equation}
\underset{0\leq t\leq T}{\sup}\|\he_N(\cdot,t)\|_{H^1_0(\Omega)}+\|\he_N\|_{L^2(0,T;H^2(\Omega))}+\|\partial_t\he_N\|_{L^2(0,T;L^2(\Omega))}\leq c_1\|\hp\|_{L^2(0,T;L^2(\Omega))},
\end{equation}
where $c_1$ only depends on $\Omega$ and $T$. It suffices to estimate the right hand side.

Denote $\tilde{L}_n(t):=L_n(\frac{2t}{T}-1)$, then $\{\tilde{L}_n\}_{n=0}^N$ is an orthogonal basis for $P^N([0,T])$, and $\{\tilde{L}_n\}_{n=0}^\infty$ is an orthogonal Schauder basis for $L^2([0,T])$. Especially, by \eqref{09} we have
\begin{equation}\label{24}
\int_0^T\tilde{L}_n(t)\tilde{L}_m(t)\td t=\frac{T}{2n+1}\delta_{mn},\quad\forall m,n\in\mathbb{N}.
\end{equation}

Note that
\begin{multline}\label{32}
\hp(x,t)=\sum_{n=1}^N\left(\hw_n(x)\partial_t\phi_n^{(1)}(t)-\Delta_x\hw_n(x)\phi_n^{(1)}(t)\right)-\pi^{N-1}_{[0,T]}f+(\pi^{N-1}_{[0,T]}f-f)\\
=\left(\sum_{n=1}^N\hw_n(x)\partial_t\phi_n^{(1)}(t)-\sum_{n=1}^{N-1}\Delta_x\hw_n(x)\phi_n^{(1)}(t)-\Delta_x\hw_N(x)\tilde{L}_{N-1}(t)-\pi^{N-1}_{[0,T]}f\right)\\
+\left(\pi^{N-1}_{[0,T]}f-f-\Delta_x\hw_N(x)\tilde{L}_N(t)\right):=I_1(x,t)+I_2(x,t).
\end{multline}
Clearly, \eqref{32} is an orthogonal decomposition on the subspace $P^{N-1}([0,T])$, namely, $I_1=\pi^{N-1}_{[0,T]}\hp$ and $I_2=\hp-\pi^{N-1}_{[0,T]}\hp$. Then we have
\begin{multline}\label{33}
\|\hp\|_{L^2(0,T;L^2(\Omega))}\\
\leq \|\pi^{N-1}_{[0,T]}\hp\|_{L^2(0,T;L^2(\Omega))}+\|\pi^{N-1}_{[0,T]}f-f\|_{L^2(0,T;L^2(\Omega))}+\|\Delta_x\hw_N(x)\tilde{L}_N(t)\|_{L^2(0,T;L^2(\Omega))}.
\end{multline}

For the first term in \eqref{33}, note that $\pi^{N-1}_{[0,T]}\hp(x,t)=\sum_{n=0}^{N-1}\tdp_n(x)\tilde{L}_n(t)$, with
\begin{equation}\label{31}
\tdp_n(x)=\frac{\int_0^T\hp(x,t)\tilde{L}_n(t)\td t}{\int_0^T|\tilde{L}_n(t)|^2\td t}=\frac{2n+1}{T}\int_0^T\hp(x,t)\tilde{L}_n(t)\td t.
\end{equation}
Also, note that $\tilde{L}_n=\sum_{j=1}^{n+1}\psi_j^{(1)}$ for $0,\cdots,N-1$, then for every $x\in\Omega$,
\begin{multline}\label{29}
\sum_{n=0}^{N-1}(2n+1)\left(\hp(x,\cdot),\tilde{L}_n\right)^2=\sum_{n=0}^{N-1}\left(\sum_{j=1}^{n+1}\sqrt{2n+1}\left(\hp(x,\cdot),\psi_j^{(1)}\right)\right)^2\\
\leq\sum_{n=0}^{N-1}(n+1)\sum_{j=1}^{n+1}(2n+1)\left(\hp(x,\cdot),\psi_j^{(1)}\right)^2=\sum_{j=1}^N\sum_{n=j-1}^{N-1}(n+1)(2n+1)\left(\hp(x,\cdot),\psi_j^{(1)}\right)^2\\
=N^3\sum_{j=1}^Nr_j\left(\hp(x,\cdot),\psi_j^{(1)}\right)^2=N^3\sum_{j=1}^Nr_j\left|R_j^{(1)}[\hw_1(x),\cdots,\hw_N(x)]\right|^2.
\end{multline}
Therefore, by \eqref{29},
\begin{multline}\label{27}
\|\pi^{N-1}_{[0,T]}\hp\|_{L^2(0,T;L^2(\Omega))}^2=\int_\Omega\left\|\sum_{n=0}^{N-1}\tdp_n(x)\tilde{L}_n(\cdot)\right\|_{L^2([0,T])}^2\td x\\
=\int_\Omega\int_0^T\sum_{n=0}^{N-1}\tdp_n^2(x)\tilde{L}_n^2(t)\td t\td x=\sum_{n=0}^{N-1}\int_\Omega \tdp_n^2(x)\td x\int_0^T\tilde{L}_n^2(t)\td t\\
\leq\sum_{n=0}^{N-1}\frac{2n+1}{T}\left\|\int_0^T\hp(x,t)\tilde{L}_n(t)\td t\right\|_{L^2(\Omega)}^2\leq \frac{N^3}{T}\sum_{j=1}^{N}r_j \left\|R_j[\hw_1,\cdots,\hw_N]\right\|_{L^2(\Omega)}^2.
\end{multline}

For the second term in \eqref{33}, using Corollary \ref{cor01} we have if $N\gg m$,
\begin{multline}\label{28}
\|\pi^{N-1}_{[0,T]}f-f\|_{L^2(0,T;L^2(\Omega))}^2\leq\int_\Omega  c_2N^{-2m}\|\partial_t^mf(x,\cdot)\|_{L_m^2([0,T])}^2\td x\\
= c_2N^{-2m} \|\partial_t^mf\|_{L_m^2(0,T;L^2(\Omega))}^2,
\end{multline}
where $c_2$ only depends on $m$ and $T$.

For the third term in \eqref{33}, using \eqref{24} we have
\begin{equation}\label{26}
\|\Delta_x\hw_N(x)\tilde{L}_N(t)\|_{L^2(0,T;L^2(\Omega))}=\|\Delta_x\hw_N\|_{L^2(\Omega)}\|\tilde{L}_N\|_{L^2([0,T])}=\sqrt{\frac{T}{2N+1}} \|\Delta_x\hw_N\|_{L^2(\Omega)}.
\end{equation}

Combining \eqref{27}, \eqref{28} and \eqref{26} completes the proof.
\end{proof}

\begin{remark}
In Theorem \ref{thm01}. the derived error bound has a quantity $O(N^\frac{3}{2})$ multiplied to the minimal loss $\sqrt{\delta}$. This quantity can also be interpreted as the condition number of the transformation matrix between the used basis $\{\psi^{(1)}_j\}_{j=1}^N$ and the orthonormal basis $\{\sqrt{\frac{2n+1}{T}}\tilde{L}_n\}_{n=0}^{N-1}$ of the test space $Y_N^{(1)}$. More precisely, we have $\left[\sqrt{\frac{1}{T}}\tilde{L}_0~\cdots~\sqrt{\frac{2N-1}{T}}\tilde{L}_{N-1}\right]^\top=D\left[\psi^{(1)}_1~\cdots~\psi^{(1)}_N\right]^\top$ with
$$
D_{ij}=\begin{cases}\sqrt{\frac{2i-1}{T}},\quad j\leq i,\\0,\quad \quad \quad \quad j>i,\end{cases}$$ and the 2-condition number of $D$ is estimated as $O(N^\frac{3}{2})$. A slightly different proof can be derived from this fact.
\end{remark}

\subsection{A priori error estimates}\label{sec_a_priori_estimate}
Now let us investigate how small the loss function $J^{\tP}$ in \eqref{22} can be minimized if we use special hypothesis space $\mathcal{S}$. Recall that $\mathcal{S}=\left\{\hw=\nu\hphi,~\hphi\in\mathcal{F}_{L,M}\right\}$ has an FNN-based structure, it is expected that $\mathcal{S}$ with larger depth $L$ or width $W$ leads to smaller optimal $J^{\tP}$. More precisely, we expect the convergence $\min_{\hw_n\in\mathcal{S}}~J^{\tP}\rightarrow0$ if $L,M\rightarrow\infty$. If this holds, then Theorem \ref{thm01} directly shows that the solution error converges to zero as $L,M,N\rightarrow\infty$ in some sense.

\subsubsection{Barron function and its approximation}
The analysis is essentially based on the approximation property of FNNs for special functions. Although neural network approximation theory has been widely developed in recent years, we will specially consider in this paper the approximation error for Barron-type functions in $H^2$ norm. It has been studied that the approximation errors of two-layer FNNs for Barron-type functions have dimension-independent rates with respect to the width \cite{Barron1992,Barron1993,E2019,E2020_2,Siegel2020_1,Siegel2020_2,Caragea2020}. We remark that stronger error bounds that overcome the curse of dimensionality can be achieved by using FNNs with advanced activation functions \cite{Shen2021_1,Shen2021_2,Shen2021_3,Jiao2021}.

Consider the following class of two-layer FNNs with $M$ neurons
\begin{equation}\label{35}
\mathcal{F}_{2,M}=\left\{\hphi(x;\theta)=\sum_{m=1}^Ma_m\sigma(w_m^\top x+b_m):~w_m\in\mathbb{R}^d,a_m,b_m\in\mathbb{R}\right\},
\end{equation}
where $\theta:=\{a_m,w_m,b_m\}_{1\leq m\leq M}$ is the set of parameters. Note the definition \eqref{35} coincides with the generic FNN \eqref{36} when $L=2$.

Also, given a function $f:\mathbb{R}^d\rightarrow\mathbb{R}$, its $s$-th order Barron norm is defined as
\begin{equation}
\|f\|_{\mathcal{B}^s}=\int_{\mathbb{R}^d}(1+|\omega|)^s|\mathcal{F}[f]|\td \omega,
\end{equation}
where
\begin{equation}
\mathcal{F}[f](\omega):=(2\pi)^{-d/2}\int_{\mathbb{R}^d}f(x)e^{-\ti \omega^\top x}\td x,
\end{equation}
is the Fourier transform of $f$. We use $\mathcal{B}^s$ to denote the Banach space of functions having bounded Barron norm, namely, the Barron space. It is shown that $\mathcal{B}^s$ is continuously embedded in $H^s(\Omega)$ if $\Omega$ is bounded \cite{Siegel2020_1}.

The following lemma shows differentiable functions with compact support and sufficiently high-order derivatives are in $\mathcal{B}^s$.
\begin{lemma}\label{lem04}
Let $s\geq1$ and $f\in C_\tc^\gamma(\mathbb{R}^d)$ with $\gamma>d/2+s$, then $f\in\mathcal{B}^s$ in the sense that
\begin{equation}
\|f\|_{\mathcal{B}^s}^2\leq c\int_{\mathbb{R}^d}|f|^2+\left|D^\gamma f\right|^2 \td x<\infty,
\end{equation}
for some constant $c$ only depending on $\gamma$ and $s$.
\end{lemma}
\begin{proof}
\begin{multline}\label{48}
\|f\|_{\mathcal{B}^s}=\int_{\mathbb{R}^d}(1+|\omega|)^s\left|\mathcal{F}[f]\right|\td\omega\\
=\int_{\mathbb{R}^d}(1+|\omega|^{2(\gamma-s)})^{-1/2}\cdot(1+|\omega|^{2(\gamma-s)})^{1/2}(1+|\omega|)^s\left|\mathcal{F}[f]\right|\td\omega.
\end{multline}
Since $\gamma>d/2+s$, $\int_{\mathbb{R}^d}(1+|\omega|^{2(\gamma-s)})^{-1}\td\omega$ is a finite number depending on $\gamma$ and $s$. So using Cauchy-Schwarz inequality on \eqref{48} leads to
\begin{equation}\label{49}
\|f\|_{\mathcal{B}^s}^2\leq c\int_{\mathbb{R}^d}(1+|\omega|^{2\gamma})\left|\mathcal{F}[f]\right|^2\td\omega.
\end{equation}
Using the identity $\mathcal{F}[\partial^\alpha_x f]=(\ti \omega)^\alpha\mathcal{F}[f]$ for each multiindex $\alpha$ satisfying $|\alpha|\leq\gamma$ and Parseval's identity, it follows \eqref{49} that
\begin{equation}
\|f\|_{\mathcal{B}^s}^2\leq c\int_{\mathbb{R}^d}\left|\mathcal{F}[f]\right|^2+\left|\mathcal{F}[D^\gamma f]\right|^2 \td\omega=c\int_{\mathbb{R}^d}|f|^2+\left|D^\gamma f\right|^2 \td x,
\end{equation}
which is finite since $f\in C_\tc^\gamma(\mathbb{R}^d)$.
\end{proof}

Next, we introduce the approximation property of FNNs, which implies any function in the Barron space can be efficiently approximated by two-layer FNNs in Sobolev norms.
\begin{lemma}[Corollary 1, \cite{Siegel2020_1}]\label{lem01}
Let $s\in\mathbb{N}$. Suppose the activation function of $\mathcal{F}_{2,M}$ is set as the sigmoidal function $\sigma(x)=(1+e^{-x})^{-1}$, arctan $\sigma(x)=\arctan(x)$, hyperbolic tangent $\sigma(x)=\tanh(x)$ or SoftPlus $\sigma(x)=\log(1+e^x)$. Then for any $f\in \mathcal{B}^{s+1}$, we have
\begin{equation}
\underset{\hphi\in\mathcal{F}_{2,M}}{\inf}\|\hphi-f\|_{H^s(\Omega)}\leq|\Omega|^{\frac{1}{2}}C(s,\text{diam}(\Omega),\sigma)\|f\|_{\mathcal{B}^{s+1}}M^{-\frac{1}{2}}.
\end{equation}
\end{lemma}

\subsubsection{Continuous extension in $\mathbb{R}^d$}
We should be careful that the PDE solution is spatially defined in a bounded domain $\Omega$ rather than $\mathbb{R}^d$. To apply the preceding approximation property, we choose to use extension approaches to extend the continuous solutions from $\overline{\Omega}\times[0,T]$ to $\mathbb{R}^d\times[0,T]$. Denote $C^k_{0^-}([0,T]):=\{f\in C^k([0,T]):f(0)=0\}$ (we omit $k$ if $k=0$). The existence of continuous extension is given by the following lemma.
\begin{lemma}\label{lem05}
Let $f(x,t)\in C^k_{0^-}(0,T;C^\gamma(\overline{\Omega}))$ with some $\gamma\geq0$, there exists some extension $f_\ext\in C^k_{0^-}(0,T;C_\tc^\gamma(\mathbb{R}^d))$ in the sense that $f_\ext(x,t)=f(x,t)$ if $(x,t)\in \overline{\Omega}\times[0,T]$.
\end{lemma}
\begin{proof}
Note that $\{\phi^{(1)}_n\}_{n=1}^\infty$ forms a Schauder basis for $C_{0^-}([0,T])$, so $f$ has an expansion $f(x,t)=\sum_{n=1}^\infty \tilde{f}_n(x)\phi^{(1)}_n(t)$ with $\tilde{f}_n(x)\in C^\gamma(\overline{\Omega})$. Now given a closed subset $\Lambda\subset\mathbb{R}^d$ satisfying $\overline{\Omega}\subset\subset\Lambda$, using continuous extension (e.g., Whitney extension theorem) we can always find some $\tilde{f}_{n,\ext}\in C_\tc^\gamma(\mathbb{R}^d)$ such that $\tilde{f}_{n,\ext}=\tilde{f}_n$ in $\overline{\Omega}$ and $\tilde{f}_{n,\ext}=0$ in $\mathbb{R}^d\backslash\Lambda$ for all $n$. Then the desired extension is formed by $f_\ext(x,t):=\sum_{n=1}^\infty \tilde{f}_{n,\ext}(x)\phi^{(1)}_n(t)$. It is clear that $f_\ext(x,t)=f(x,t)$ in $\overline{\Omega}\times[0,T]$. Moreover, $f_\ext(x,0)=0$ for all $x\in\overline{\Omega}$, and $f_\ext$ is $C^k$ in $t$ follows that $f(x,t)$ is $C^k$ in $t$.
\end{proof}

Next, let us define a special norm for functions in $C(0,T;C_\tc^\gamma(\mathbb{R}^d))$. Recall that $\{\tilde{L}_n:=L(\frac{2t}{T}-1)\}_{n=0}^\infty$ forms a Schauder basis for $C([0,T])$. So, given $f\in C(0,T;C_\tc^\gamma(\mathbb{R}^d))$, $f$ has an expansion $f(x,t)=\sum_{n=0}^\infty \tilde{f}_n(x)\tilde{L}_n(t)$. We define
\begin{equation}
\|f\|_{L^2_*}:=\left(\sum_{n=0}^\infty\int_{\mathbb{R}^d}|\tilde{f}_n|^2 \td x\right)^{1/2}.
\end{equation}
Note that $\{\tilde{f}_n\}$ are the coefficients associated with the $L^2$-orthogonal basis $\{\tilde{L}_n\}$ in time variable, so $\|\cdot\|_{L^2_*}$ is actually a weighted $L^2$ norm over $\mathbb{R}^d\times[0,T]$ (but not the standard $L^2$ norm because $\{\tilde{L}_n\}$ is not orthonormal). We show that $\|f\|_{L^2_*}$ is always bounded by the standard $L^2$ norm of $\partial_t^2f$ if $f$ is $C^2$ smooth in $t$ by the following lemma.
\begin{lemma}\label{lem03}
Suppose $f\in C^2(0,T;C_\tc(\mathbb{R}^d))$ with $\gamma\geq0$, it holds that
\begin{equation}\label{59}
\|f\|_{L^2_*}^2\leq\frac{5T^3}{4}\|\partial_t^2f\|_{L^2(\mathbb{R}^d\times[0,T])}^2<\infty.
\end{equation}
\end{lemma}
\begin{proof}
By the orthogonality relation \eqref{24}, we have
\begin{equation}\label{52}
\tilde{f}_n=\frac{\int_0^T f(x,t)\tilde{L}_n(t)\td t}{\int_0^T|\tilde{L}_n(t)|^2\td t}=\frac{2n+1}{T}\int_0^T f(x,t)L_n(\frac{2t}{T}-1)\td t.
\end{equation}
Using the identity \eqref{62} and integration by parts, we can obtain
\begin{multline}\label{50}
\left|\int_0^T f(x,t)L_n(\frac{2t}{T}-1)\td t\right|=\left|\frac{T}{4n+2}\int_0^T\partial_t f(x,t)\left(L_{n+1}(\frac{2t}{T}-1)-L_{n-1}(\frac{2t}{T}-1)\right)\td t\right|\\
\leq\frac{T}{4n+2}\left(\left|\int_0^T\partial_t f(x,t)L_{n+1}(\frac{2t}{T}-1)\td t\right|+\left|\int_0^T\partial_t f(x,t)L_{n-1}(\frac{2t}{T}-1)\td t\right|\right).
\end{multline}
Following \eqref{50}, we repeat doing the preceding step with integration by parts on the last two terms, then it leads to
\begin{equation}\label{53}
\left|\int_0^T f(x,t)L_n(\frac{2t}{T}-1)\td t\right|\leq 4\left(\frac{T}{4n+2}\right)^2\int_0^T|\partial_t^2 f(x,t)|\td t,
\end{equation}
in which we use the boundedness $|L_n(x)|\leq1$ for all $x\in[-1,1]$ and all $n$. Combining \eqref{52} and \eqref{53} and using Cauchy-Schwarz inequality, we have
\begin{equation}
|\tilde{f}_n|^2\leq\frac{T^3}{(2n+1)^2}\int_0^T|\partial_t^2 f(x,t)|^2\td t,
\end{equation}
which completes the proof with the relation $\sum_{n=0}^\infty(2n+1)^{-2}<5/4$.
\end{proof}

Using the concept of continuous extension and the $L^2_*$ norm, we can define another norm for functions in $C_{0^-}(0,T;C^\gamma(\overline{\Omega}))$ that will be used in later analysis. Specifically, for each $f\in C_{0^-}(0,T;C^\gamma(\overline{\Omega}))$, we define
\begin{equation}\label{54}
\|f\|_\dagger:=\inf_{f_\ext}\left(\|f_\ext\|_{L^2_*}^2+\|D^\gamma f_\ext\|_{L^2_*}^2\right)^{1/2},
\end{equation}
where the infimum is taken over all possible extensions $f_\ext$ of $f$ in the space $C_{0^-}(0,T;C_\tc^\gamma(\mathbb{R}^d))$. By Lemma \ref{lem05}, such extension always exists, so $\|f\|_\dagger$ is well-defined. And by Lemma \ref{lem03}, $\|f\|_\dagger$ is finite if $f\in C^2(0,T;C^\gamma(\overline{\Omega}))$.

\subsubsection{Error estimates}
Now let us estimate how small $J^\tP$ can be minimized in the hypothesis space $\mathcal{S}=\left\{\hw=\nu\hphi,~\hphi\in\mathcal{F}_{2,M}\right\}$, and thereafter derive the final solution error. Before presenting the estimate, we discuss the existence of the boundary governing function $\nu(x)$. The following result shows we can always find some $\nu$ that is positive everywhere and vanishes on the boundary if the domain is smooth enough.

\begin{lemma}\label{lem02}
Let $\Omega$ be a $C^\gamma$ domain in $\mathbb{R}^d$ with $\gamma\geq2$. Then there exists some $\nu\in C_0^\gamma(\overline{\Omega})$ such that $\nu>0$ in $\Omega$.
\end{lemma}
\begin{proof}
We give a constructive proof. Consider the elliptic equation
\begin{equation}
\begin{cases}-\Delta_x \nu(x)=1,\quad x\in\Omega,\\\nu(x)=0,\quad x\in\partial\Omega,\end{cases}
\end{equation}
which admits a unique solution since $\Omega$ is $C^\gamma$ smooth. By the global regularity theorem for strictly elliptic equations (e.g., Theorem 6.19, Page. 111 in \cite{Gilbarg2015}), the solution $\nu$ is $C^\gamma$ smooth in $\overline{\Omega}$. And by the maximum principle of strictly elliptic operators, $\nu$ only takes its minimum on $\partial\Omega$, so $\nu>0$ in $\Omega$.
\end{proof}

For simplicity, we can normalize $\nu$ by multiplying some positive constant so that
\begin{equation}\label{37}
\max\{\|\nu\|_\infty,\||\nabla_x\nu|^2\|_\infty,\|\Delta_x\nu\|_\infty\}=1.
\end{equation}

Now we give the following error estimates.
\begin{theorem}\label{thm02}
Let $\gamma>2/d+3$ be an integer. Let $u$ be the classical solution of \eqref{02} with $\mathcal{L}_x=-\Delta_x$ and $\Omega$ being $C^\gamma$ smooth. Let $\nu\in C_0^\gamma(\overline{\Omega})$ satisfy $\nu>0$ in $\Omega$ and the normalization condition \eqref{37} (Lemma \ref{lem02} implies the existence of $\nu$). Moreover, let $m,m'\geq0$ be two integers. Suppose $f\in\mathcal{W}^m(0,T;L^2(\Omega))$, $u\in C^{\max(2,m')}(0,T;C_0^\gamma(\overline{\Omega}))$ and $u/\nu<\infty$ on $\partial\Omega$ for each $t\in[0,T]$. If $\{\hw_n\}_{n=1}^N$ solves the minimization \eqref{22} with hypothesis space $\mathcal{S}=\left\{\hw=\nu\hphi,~\hphi\in\mathcal{F}_{2,M}\right\}$, where the activation function $\sigma$ of $\mathcal{F}_{2,M}$ satisfies the hypothesis of Lemma \ref{lem01}. then $\hu_N(x,t):=\sum_{n=1}^N\hw_n(x)\phi_n^{(1)}(t)$ satisfies the estimate
\begin{multline}\label{58}
\underset{0\leq t\leq T}{\sup}\|u(\cdot,t)-\hu_N(\cdot,t)\|_{H^1_0(\Omega)}+\|u-\hu_N\|_{L^2(0,T;H^2(\Omega))}+\|\partial_t(u-\hu_N)\|_{L^2(0,T;L^2(\Omega))}\\
\leq c\left(d^\frac{1}{2}N^2M^{-\frac{1}{2}}\|u/\nu\|_\dag +d^\frac{1}{2}N^{-\frac{3}{2}-m'}\|\partial_t^{m'}(u/\nu)\|_{L_{m'}^2(0,T;H^2(\Omega))}+N^{-m} \|\partial_t^m f\|_{L_m^2(0,T;L^2(\Omega))}\right)
\end{multline}
provided that $N\gg m,m'$, where $c$ only depends on $\Omega$, $T$, $\gamma$, $\sigma$, $\lambda$, $m$ and $m'$.
\end{theorem}

\begin{proof}
We use the notation $A\lesssim B$ to mean $A\leq cB$ for some constant $c$ only depending on $\Omega$, $T$, $\gamma$, $\sigma$, $\lambda$, $m$ and $m'$.

Denote $Q_T:=\overline{\Omega}\times[0,T]$. Since $u(x,0)=0$ in $\overline{\Omega}$ and $\{\phi^{(1)}_n\}_{n=1}^\infty$ forms a Schauder basis for $C^2_{0^-}([0,T])$, we can expand $u$ in $Q_T$ as $u=\sum_{n=1}^\infty \tdu_n(x)\phi^{(1)}_n(t)$ with $\tdu_n\in C_0^\gamma(\overline{\Omega})$. Denote $\tdu^*_n(x):=\tdu_n(x)/\nu(x)$, then the division $u/\nu=\sum_{n=1}^\infty \tdu^*_n(x)\phi^{(1)}_n(t)$ is well-defined and continuous in $Q_T$ since $u/\nu$ is finite on $\partial\Omega$ for each $t$ and $\nu\neq0$ in $\Omega$. By the hypothesis $\nu\in C_0^\gamma(\overline{\Omega})$, we have $\tdu^*_n\in C^\gamma(\overline{\Omega})$ and hence $u/\nu\in C^2_{0^-}(0,T;C^\gamma(\overline{\Omega}))$.

Now from \eqref{54}, let $u^*_\ext\in C^2_{0^-}(0,T;C_\tc^\gamma(\mathbb{R}^d))$ be some extension of $u/\nu$ such that
\begin{equation}
\|u/\nu\|_\dagger=\left(\|u^*_\ext\|_{L^2_*}^2+\|D^\gamma u^*_\ext\|_{L^2_*}^2\right)^{1/2},
\end{equation}
which is finite by Lemma \ref{lem03}. We expand $u^*_\ext=\sum_{n=1}^\infty\tdu^*_{n,\ext}(x)\phi^{(1)}_n(t)$ with $\tdu^*_{n,\ext}\in C_\tc^\gamma(\mathbb{R}^d)$. It is clear that $\tdu^*_{n,\ext}$ is an extension of $\tdu^*_n$ from $\overline{\Omega}$ to $\mathbb{R}^d$.

Next by Lemma \ref{lem04} and \ref{lem01}, there exists $\hpsi_n\in\mathcal{F}_{2,M}$ such that
\begin{equation}\label{39}
\|\hpsi_n-\tdu^*_n\|_{H^2(\Omega)}^2=\|\hpsi_n-\tdu^*_{n,\ext}\|_{H^2(\Omega)}^2\lesssim M^{-1}\|\tdu^*_{n,\ext}\|_{\mathcal{B}^3}^2\lesssim M^{-1} \int_{\mathbb{R}^d}|\tdu^*_{n,\ext}|^2+\left|D^\gamma \tdu^*_{n,\ext}\right|^2 \td x,
\end{equation}
for $n=1\cdots,N$. Then it follows \eqref{39} that
\begin{multline}\label{56}
\sum_{n=1}^N\|\hpsi_n-\tdu^*_n\|_{H^2(\Omega)}^2\lesssim M^{-1}\sum_{n=1}^\infty \int_{\mathbb{R}^d}\left|\tdu^*_{n,\ext}\right|^2+\left|D^\gamma \tdu^*_{n,\ext}\right|^2 \td x\\
\lesssim M^{-1}\sum_{n=1}^\infty \int_{\mathbb{R}^d}\left|\tdu^*_{n,\ext}+\tdu^*_{n+1,\ext}\right|^2+\left|D^\gamma \tdu^*_{n,\ext}+D^\gamma \tdu^*_{n+1,\ext}\right|^2 \td x,
\end{multline}
where the last inequality follows the fact $\sum_{n=1}^\infty|r_n+r_{n+1}|^2-\sum_{n=1}^\infty|r_n|^2=\left|\sum_{n=1}^\infty r_n\right|^2\geq0$ for all real $r_n$.

On the other hand, using \eqref{55} we have
\begin{equation}
u^*_\ext=\tdu^*_{1,\ext}(x)\tilde{L}_0(t)+\sum_{n=1}^\infty\left(\tdu^*_{n,\ext}(x)+\tdu^*_{n+1,\ext}(x)\right)\tilde{L}_n(t),
\end{equation}
so it follows \eqref{56} that
\begin{equation}
\sum_{n=1}^N\|\hpsi_n-\tdu^*_n\|_{H^2(\Omega)}^2\lesssim M^{-1}\left(\|u^*_\ext\|_{L^2_*}^2+\|D^\gamma u^*_\ext\|_{L^2_*}^2\right)=M^{-1}\|u/\nu\|_\dagger^2.
\end{equation}

Since $\nu\hpsi_n\in\mathcal{S}$, substituting $f=u_t-\Delta_xu=\sum_{n=1}^\infty\nu\tdu^*_n(x)\partial_t\phi^{(1)}_n(t)-\Delta_x(\nu\tdu^*_n(x))\phi^{(1)}_n(t)$ in $J^\tP$ and using Proposition \ref{prop01}, we obtain
\begin{multline}\label{38}
J^\tP[\hw_1,\cdots,\hw_N]\leq J^\tP[\nu\hpsi_1,\cdots,\nu\hpsi_N]\\
=\sum_{j=1}^Nr_j\left\|\sum_{n=1}^Na_{jn}^{(1)}(\nu\hpsi_n-\nu\tdu^*_n)-\sum_{n=1}^Nb_{jn}^{(1)}(\Delta_x(\nu\hpsi_n)-\Delta_x(\nu\tdu^*_n))\right\|_{L^2(\Omega)}^2+\lambda N^{-4}\|\Delta_x(\nu\hpsi_N)\|_{L^2(\Omega)}^2\\
\lesssim ~N\sum_{n=1}^N\|\nu\hpsi_n-\nu\tdu^*_n\|_{L^2(\Omega)}^2+N\|\Delta_x(\nu\hpsi_1)-\Delta_x(\nu\tdu^*_1)\|_{L^2(\Omega)}^2\\
+\sum_{n=1}^N\|\Delta_x(\nu\hpsi_n)-\Delta_x(\nu\tdu^*_n)\|_{L^2(\Omega)}^2+N^{-4}\left(\|\Delta_x(\nu\hpsi_N)-\Delta_x(\nu\tdu^*_N)\|_{L^2(\Omega)}+\|\Delta_x(\nu\tdu^*_N)\|_{L^2(\Omega)}\right)^2\\
\lesssim ~N\sum_{n=1}^N\|\nu\hpsi_n-\nu\tdu^*_n\|_{L^2(\Omega)}^2+N\|\Delta_x(\nu\hpsi_1)-\Delta_x(\nu\tdu^*_1)\|_{L^2(\Omega)}^2\\
+\sum_{n=1}^N\|\Delta_x(\nu\hpsi_n)-\Delta_x(\nu\tdu^*_n)\|_{L^2(\Omega)}^2+N^{-4}\|\Delta_x(\nu\tdu^*_N)\|_{L^2(\Omega)}^2.
\end{multline}
On one hand, \eqref{37} implies that
\begin{gather*}
\|\nu\hpsi_n-\nu\tdu^*_n\|_{L^2(\Omega)}^2\leq\|\hpsi_n-\tdu^*_n\|_{L^2(\Omega)}^2,\\
\|\Delta_x(\nu\hpsi_n)-\Delta_x(\nu\tdu^*_n)\|_{L^2(\Omega)}^2\leq d\|\hpsi_n-\tdu^*_n\|_{H^2(\Omega)}^2,
\end{gather*}
then using \eqref{39} we have
\begin{multline}\label{41}
N\sum_{n=1}^N\|\nu\hpsi_n-\nu\tdu^*_n\|_{L^2(\Omega)}^2+N\|\Delta_x(\nu\hpsi_1)-\Delta_x(\nu\tdu^*_1)\|_{L^2(\Omega)}^2
\\+\sum_{n=1}^N\|\Delta_x(\nu\hpsi_n)-\Delta_x(\nu\tdu^*_n)\|_{L^2(\Omega)}^2
\lesssim dN\sum_{n=1}^N\|\hpsi_n-\tdu^*_n\|_{H^2(\Omega)}^2
\lesssim dNM^{-1}\|u/\nu\|_\dagger^2.
\end{multline}
On the other hand, we note that for $k=0,1,2$,
\begin{multline}
\partial_x^k(\frac{u}{\nu})=\sum_{n=1}^\infty\partial_x^k\tdu^*_n(x)\phi^{(1)}_n(t)=\sum_{n=1}^\infty\partial_x^k\tdu^*_n(x)\left(\tilde{L}_n(t)+\tilde{L}_{n-1}(t)\right)\\
=\partial_x^k\tdu^*_1(x))\tilde{L}_0(t)+\sum_{n=1}^\infty\left(\partial_x^k\tdu^*_n(x)+\partial_x^k\tdu^*_{n+1}(x)\right)\tilde{L}_n(t).
\end{multline}
So for every $x\in\Omega$,
\begin{multline}
\|\pi^{N-1}_{[0,T]}\partial_x^k(\frac{u(x,\cdot)}{\nu(x)})-\partial_x^k(\frac{u(x,\cdot)}{\nu(x)})\|_{L^2([0,T])}^2=\int_0^T\sum_{n=N}^\infty|\partial_x^k\tdu^*_n(x)+\partial_x^k\tdu^*_{n+1}(x)|^2\cdot|\tilde{L}_n(t)|^2\td t\\=\sum_{n=N}^\infty\frac{T}{2n+1}|\partial_x^k\tdu^*_n(x)+\partial_x^k\tdu^*_{n+1}(x)|^2\geq\frac{T}{2N+1}|\partial_x^k\tdu^*_N(x)|^2.
\end{multline}
Since $u\in C^{m'}(0,T;C_0^\gamma(\overline{\Omega}))\subset\mathcal{W}^{m'}(0,T;H^2(\overline{\Omega}))$, by Corollary \ref{cor01}, if $N\gg m'$,
\begin{multline}
|\partial_x^k\tdu^*_N(x)|^2\lesssim N\|\pi^{N-1}_{[0,T]}\partial_x^k(\frac{u(x,\cdot)}{\nu(x)})-\partial_x^k(\frac{u(x,\cdot)}{\nu(x)})\|_{L^2([0,T])}^2\lesssim N^{1-2m'}\|\partial_t^{m'}\partial_x^k(\frac{u(x,\cdot)}{\nu(x)})\|_{L_{m'}^2([0,T])}^2,
\end{multline}
and therefore using \eqref{37} we have
\begin{multline}\label{42}
\|\Delta_x(\nu\tdu^*_N)\|_{L^2(\Omega)}^2\leq\|\tdu^*_N\|_{L^2(\Omega)}^2+\|\nabla_x\tdu^*_N\|_{L^2(\Omega)}^2+\|\Delta_x\tdu^*_N\|_{L^2(\Omega)}^2\\
\lesssim \left(\|\partial_t^{m'}(u/\nu)\|_{L_{m'}(0,T;L^2(\Omega))}^2+\|\nabla_x\partial_t^{m'}(u/\nu)\|_{L_{m'}^2(0,T;L^2(\Omega))}^2+d\sum_{i=1}^d\|\partial_{x_i}^2\partial_t^{m'}(u/\nu)\|_{L_{m'}(0,T;L^2(\Omega))}^2\right)N^{1-2m'}\\
\lesssim dN^{1-2m'}\|\partial_t^{m'}(u/\nu)\|_{L_{m'}^2(0,T;H^2(\Omega))}^2.
\end{multline}
Finally, combining \eqref{38}, \eqref{41}, \eqref{42}, and using Theorem \ref{thm01} completes the proof.
\end{proof}

Theorem \ref{02} implies that the solution error approaches to zero if $M^{1/2}N^{-2}\rightarrow\infty$ and $N\rightarrow\infty$ given that $m\geq1$ and $m'\geq0$. It requires that $M=O(N^p)$ with $p>4$. So in theory the width of the neural network should be much larger than the number of basis in the Galerkin formulation. However, in practical examples (e.g., in Section \ref{sec_examples}), the width getting small errors is much less than the theoretical requirement.

\subsubsection{Long time behavior estimates}\label{sec:long_time_estimates}
Let us further investigate how the solution error behaves as $T$ varies. We do rescale the time variable by taking $t'=t/T$ in \eqref{02}, then the problem becomes
\begin{equation}\label{57}
\begin{cases}
\frac{1}{T}\frac{\partial u(x,Tt')}{\partial t'}+\mathcal{L}_xu(x,Tt')=f(x,Tt'),\quad\text{in}~\Omega\times(0,1],\\
u(x,0) = 0,\quad\text{in}~\Omega,\\
u(x,t') = 0,\quad\text{on}~\partial\Omega\times[0,1].
\end{cases}
\end{equation}
Hence it suffices to consider the problem \eqref{57} in standard time interval with $t$ multiplied by $T$ in the equation. For wave-like solutions oscillatory in time, it is equivalent to increase their frequency by $T$ times. Therefore, we specify the solution as the sine Fourier mode $u=\nu\sin(Tt)$ and investigate how the frequency $T$ influences the error bound in \eqref{58} for the standard heat equation in $\Omega\times[0,1]$.

For the first term, by the estimate \eqref{59} we have $\|u/\nu\|_\dagger=\|\sin(Tt)\|_\dagger\sim O(T^{7/2})$. For the second term, since
\begin{equation}
|\partial_t^{m'}(u/\nu)|=\begin{cases}T^{m'}|\cos(T t)|,\quad m'\text{~is odd},\\T^{m'}|\sin(T t)|,\quad m'\text{~is even};\end{cases},
\end{equation}
we have $\|\partial_t^{m'}(u/\nu)\|_{L_{m'}(0,T;H^2(\Omega))}\sim O(T^{m'})$. For the third term, note $f=T\nu\cdot\cos(Tt)-\Delta_x\nu\cdot\sin(Tt)$ and
\begin{equation}
|\partial_t^{m}f|\leq\begin{cases}T^{m+1}|\nu||\sin(Tt)|+T^m|\Delta_x\nu||\cos(Tt)|,\quad m\text{~is odd},\\T^{m+1}|\nu||\cos(Tt)|+T^m|\Delta_x\nu||\sin(Tt)|,\quad m\text{~is even},\end{cases}
\end{equation}
we have $\|\partial_t^m f\|_{L_m(0,T;L^2(\Omega))^2}\sim O(T^{m+1})$. Combining the preceding results shows that the error bound in \eqref{58} is dominated by $O\left(d^\frac{1}{2}N^2M^{-\frac{1}{2}}T^\frac{7}{2}++d^\frac{1}{2}N^{-\frac{3}{2}}(\frac{T}{N})^{m'}+T(\frac{T}{N})^m\right)$.

Consequently, to make the error bound small, it is necessary to set $\frac{T}{N}<1$; namely, the number of temporal basis functions should exceed the frequency of the solution. This conclusion is trivial if we note the fact that if the solution has high-frequency modes in $t$, the trial space $X_N^{(1)}$ should be large enough to contain them. Moreover, in the case $u$ is $C^\infty$ smooth in $t$, the regularity orders $m'$ and $m$ are infinitely large, and hence the error bound grows exponentially with $T$ if $M$ and $N$ are fixed.

\section{Hyperbolic equations}
Similar to the DABG method for parabolic equations introduced in Section \ref{sec_method_parabolic}, we also develop the method for the following type of hyperbolic equations.
\begin{equation}\label{19}
\begin{cases}
\frac{\partial^2 u(x,t)}{\partial t^2}+\mathcal{L}_xu(x,t)=f(x,t),\quad\text{in}~\Omega\times(0,T],\\
u(x,0) = 0,~\partial_tu(x,0) = g_0(x),\quad\text{in}~\Omega,\\
u(x,t) = 0,\quad\text{on}~\partial\Omega\times[0,T].
\end{cases}
\end{equation}
Since all discussions inherit from Section \ref{sec_method_parabolic} directly, we briefly present the definitions and method without many explanations. Also, we exhibit the error estimates that can be proved by the same argument as in Theorem \ref{thm01} and \ref{thm02}.

Recall $X^{(2)}$, $Y^{(2)}$, $Y_N^{(2)}$ are defined in Section \ref{sec_second_order_equation}, we define the trial space
\begin{equation}
\mathcal{U}^\tH:=\left\{u(x,t):u(x,\cdot)\in X^{(2)}~\forall x\in\Omega,~u|_{\partial\Omega\times[0,T]}=0\right\},
\end{equation}
where the label ``H" is short for the hyperbolic equation. Following the method for the second-order problem in Section \ref{sec_second_order_equation}, the variational form for \eqref{19} is characterized as
\begin{equation}\label{43}
\begin{cases}
\text{Find}~u\in \mathcal{U}^\tH~\text{such that}\\
\left(\partial_tu(x,\cdot), \partial_tv\right)-\alpha\left(\mathcal{L}_xu(x,\cdot), v\right)=-\left(f(x,\cdot), v\right)-g_0(x)v(0)~\forall v\in Y^{(2)},~\forall x\in\Omega.
\end{cases}
\end{equation}

Similarly, define the approximate trial space
\begin{equation}
\mathcal{U}_N^\tH:=\left\{\hu_N(x,t)=\sum_{n=1}^{N-1}\hw_n(x)\phi_n^{(2)}(t),~\hw_n\in\mathcal{S}\right\}.
\end{equation}
Then the Galerkin formulation for \eqref{19} is developed as
\begin{equation}
\begin{cases}
\text{Find}~\hu_N\in \mathcal{U}_N^\tH~\text{such that}\\
\left(\partial_t\hu_N(x,\cdot), \partial_tv_N\right)-\alpha\left(\mathcal{L}_x\hu_N(x,\cdot), v_N\right)=-\left(f(x,\cdot), v_N\right)-g_0(x)v_N(0)~\forall v_N\in Y_N^{(2)},~\forall x\in\Omega,
\end{cases}
\end{equation}
where the integration by parts is used. Taking $\hu_N=\sum_{n=1}^N\hw_n\phi_n^{(2)}$ and $v_N=\psi_j^{(2)}$ for $j=1,\cdots,N$, we obtain
\begin{equation}\label{21}
\sum_{n=1}^N\left(a_{jn}^{(2)}\hw_n(x)-b_{jn}^{(2)}\mathcal{L}_x\hw_n(x)\right)=-\left(f(x,\cdot), \psi_j^{(2)}\right)-g_0(x)\psi_j^{(2)}(0),\quad j=1,\cdots,N,
\end{equation}
with unknowns $\{\hw_n\}\subset\mathcal{S}$, where $a_{jn}^{(2)}$ and $b_{jn}^{(2)}$ are defined in \eqref{20}. Practically, we solve \eqref{21} by the following least-squares optimization,
\begin{equation}\label{34}
\underset{\hw_n\in\mathcal{S}}{\min}~J^{\tH}[\hw_1,\cdots,\hw_N]:=\sum_{j=1}^N\left\|R_j^{(2)}[\hw_1,\cdots,\hw_N]\right\|_{L^2(\Omega)}^2+\lambda N^{-5}I[\hw_N],
\end{equation}
where
\begin{equation}
R_j^{(2)}[\hw_1,\cdots,\hw_N]:=\sum_{n=1}^N\left(a_{jn}^{(2)}\hw_n(x)-b_{jn}^{(2)}\mathcal{L}_x\hw_n(x)\right)+\left(f(x,\cdot),\psi_j^{(2)}\right)+g_0(x)\psi_j^{(2)}(0)
\end{equation}
is the residual of the $j$-th equation in \eqref{21} and
\begin{multline}
I[\hw_N]:=\left(\hu_N''(x,\cdot)+\mathcal{L}_x\hu_N(x,\cdot)-f(x,\cdot),\tilde{L}_N\right)\\
=\frac{T(N-1)}{N(2N+1)}\mathcal{L}_x\hw_N(x)-\int_0^Tf(x,t)L_N(\frac{2t}{T}-1)\td t
\end{multline}
is the inner product of the differential equation residual and the orthogonal polynomial $\tilde{L}_N(t):=L_N(\frac{2t}{T}-1)$ in $L^2([0,T])$. The term $\lambda N^{-5}I[\hw_N]$ can be viewed as a regularization of the last coefficient $\hw_N$, which is specifically constructed according to the a posterior estimate. By Proposition \ref{prop02}, $R_j^{(2)}[\hw_1,\cdots,\hw_N]$ only includes at most 7 nonzero terms if $j\geq2$. Similar to the parabolic case, the regularization term $\lambda N^{-5}I[\hw_N]$ are specially chosen such that the a posterior estimate holds. More precisely, we have the following a posterior estimate for the wave equation.

\begin{theorem}\label{thm03}
Let $u$ be the classical solution of \eqref{19} with $\mathcal{L}_x=-\Delta_x$. Suppose that $f\in\mathcal{W}^m(0,T;L^2(\Omega))$, $u\in L^2(0,T;H_0^1(\Omega))$, $\partial_tu\in L^2(0,T;L^2(\Omega))$ and $\partial_{tt}u\in L^2(0,T;H^{-1}(\Omega))$. Let $\hu_N(x,t)=\sum_{n=1}^N\hw_n(x)\phi_n^{(2)}(t)$, where $\{\hw_n\}_{n=1}^N$ solves the minimization \eqref{34}. Suppose that $\hw_n\in H_0^1(\Omega)$. Denote $\delta=J^{\tH}[\hw_1,\cdots,\hw_N]$ as the minimized loss. Then
\begin{multline}
\underset{0\leq t\leq T}{\sup}\left(\|u(\cdot,t)-\hu_N(\cdot,t)\|_{H^1_0(\Omega)}+\|\partial_tu(\cdot,t)-\partial_t\hu_N(\cdot,t)\|_{L^2(\Omega)}\right)\\
\leq c\left(N^3\sqrt{\delta}+N^{-m} \|\partial_t^m f\|_{L_m(0,T;L^2(\Omega))}\right),
\end{multline}
where $c$ is a constant only depending on $\Omega$, $m$, $T$ and $\lambda$.
\end{theorem}

\begin{remark}
In Theorem \ref{thm03}, the error bound has a quantity $O(N^3)$ appended to the minimal loss $\sqrt{\delta}$, which is larger than the parabolic case. It is due to the usage of three-term combination in \eqref{45} to construct $\psi^{(2)}_n$, and the transformation matrix has a larger 2-condition number $O(N^3)$.
\end{remark}

Similar to Theorem \ref{thm02}, we also have the error estimate.
\begin{theorem}\label{thm04}
Let $\gamma>2/d+3$ be an integer. Let $u$ be the classical solution of \eqref{19} with $\mathcal{L}_x=-\Delta_x$ and $\Omega$ being $C^\gamma$ smooth. Let $\nu\in C_0^\gamma(\overline{\Omega})$ satisfy $\nu>0$ in $\Omega$ and the normalization condition \eqref{37} (Lemma \ref{lem02} implies the existence of $\nu$). Moreover, let $m\geq0$ be an integer. Suppose $f\in\mathcal{W}^m(0,T;L^2(\Omega))$, $u\in C^2(0,T;C_0^\gamma(\overline{\Omega}))$ and $u/\nu<\infty$ on $\partial\Omega$ for each $t\in[0,T]$. If $\{\hw_n\}_{n=1}^N$ solves the minimization \eqref{34} with hypothesis space $\mathcal{S}=\left\{\hw=\nu\hphi,~\hphi\in\mathcal{F}_{2,M}\right\}$, where the activation function $\sigma$ of $\mathcal{F}_{2,M}$ satisfies the hypothesis of Lemma \ref{lem01}. then $\hu_N(x,t):=\sum_{n=1}^N\hw_n(x)\phi_n^{(2)}(t)$ satisfies the estimate
\begin{multline}
\underset{0\leq t\leq T}{\sup}\left(\|u(\cdot,t)-\hu_N(\cdot,t)\|_{H^1_0(\Omega)}+\|\partial_tu(\cdot,t)-\partial_t\hu_N(\cdot,t)\|_{L^2(\Omega)}\right)\\
\leq c\left(d^\frac{1}{2}N^\frac{7}{2}M^{-\frac{1}{2}}\|u/\nu\|_\dag +N^{-m} \|\partial_t^m f\|_{L_m(0,T;L^2(\Omega))^2}\right)
\end{multline}
provided that $N\gg m$, where $c$ only depends on $\Omega$, $T$, $\gamma$, $\sigma$, $\lambda$ and $m$.
\end{theorem}

As mentioned in Section \ref{sec_a_priori_estimate}, the error bound given in Theorem \ref{thm04} is merely derived from a theoretical perspective. The width $M$ used in practice leading to convergence could be much smaller than the theoretical quantity $N^p$ with $p>7$. A hyperbolic example is presented in Section \ref{sec_case5} to demonstrate it.

\section{Numerical examples}\label{sec_examples}
\subsection{Overview}
The proposed DABG method is tested by several examples including parabolic and hyperbolic equations. Meanwhile, the popular deep least-squares (DLS) method \eqref{44} is implemented for comparison. To be consistent with the hypothesis of the preceding theory, we set homogeneous initial and boundary conditions for all examples. Therefore, we set the hypothesis space $\mathcal{N}'$ of the DLS method as
\begin{gather}
\mathcal{N}'=\left\{\hu(x,t)=t^\beta\nu(x)\hphi(x,t),~\hphi\in\mathcal{F}_{L,M}:\mathbb{R}^{d+1}\rightarrow\mathbb{R}\right\},\label{46}\\
\text{with}\quad\beta=\begin{cases}1\quad\text{for parabolic equations},\\2\quad\text{for hyperbolic equations},\end{cases}\notag
\end{gather}
where the variable $x$ and $t$ are concatenated as one input variable of the FNN $\hphi$. In \eqref{46}, $\nu$ is the same boundary governing function as DABG that is chosen corresponding to the domain shape. In our experiments, we choose $\nu(x)=(x_1^2-1)(x_2^2-1)$ for Case 1-2, where $\Omega$ is set as the square $[-1,1]^2$, and choose $\nu(x)=|x|^2-1$ for Case 3-5, where $\Omega$ is set as the unit sphere. It is clear that $\hu|_{\partial\Omega\times[0,T]}=0$ and $\partial_t^{\beta-1}\hu|_{\Omega\times\{t=0\}}=0$ for all $\hu\in\mathcal{N}'$, so the functions in $\mathcal{N}'$ always satisfy the homogeneous initial and boundary conditions and hence $\mathcal{N}'$ can be taken as the hypothesis space of the DLS formulation \eqref{44}. Other settings are summarized as follows.

\begin{itemize}
  \item {\em Environment.}
  The methods are implemented mainly in Python environment. PyTorch library with CUDA toolkit is utilized for neural network implementation and GPU-based parallel computing.
  \item {\em Optimizer and hyper-parameters.}
  The network-based optimizations problems \eqref{22} and \eqref{34} are solved by the batch gradient descent as well as the compared DLS method. For both methods, the gradient descent is implemented for totally 20000 iterations with the same decaying learning rates.
  \item {\em Network setting.}
  For DABG, we fix $L=3$ in the class $S$ and test the method for various width $M$ and basis numbers $N$. For DLS, we implement it with various depth $L$ and width $M$. For both methods, we choose the sigmoidal function as the activation function of the involved FNNs. And their parameters are initialized by
  \begin{equation}
  a, W_l, b_l\sim U(-\sqrt{M},\sqrt{M}),\quad l=1,\cdots,L.
  \end{equation}
  \item {\em Numerical integration.}
  For the evaluation of $L^2$ norm over $\Omega$, we adopt the quasi-Monte Carlo method with Halton sequence. In every iteration, the integral is approximately evaluated as the mean of the integrand at $10^4$ uniform random points.
  \item {\em Testing set and error evaluation.}
  We generate a spatial testing set $\mathcal{X}$ consisting of $N_\mathcal{X}$ uniform random points in $\Omega$. Also, denote $\mathcal{T}=\{kT/N_\mathcal{T}\}_{k=1}^{N_\mathcal{T}}$ as the set of $N_\mathcal{T}$ equidistant grid points in $[T/N_\mathcal{T},T]$. We compute the following relative $\ell^2$ error over $\mathcal{X}\times\mathcal{T}$.
  \begin{equation}
  e_{\ell^2(\mathcal{X}\times\mathcal{T})}:=\left(\sum_{t\in\mathcal{T}}\sum_{x\in\mathcal{X}}|\hu(x,t)-u(x,t)|^2 / \sum_{t\in\mathcal{T}}\sum_{x\in\mathcal{X}}|u(x,t)|^2\right)^\frac{1}{2},
  \end{equation}
  where $\hu$ and $u$ are the numerical solution and true solution, respectively.
\end{itemize}

\subsection{Case 1: 2-D heat equation oscillatory in time}\label{sec_case1}
We consider the heat equation
\begin{equation}\label{heat_equ}
\begin{cases}u_t-\Delta u = f,\quad\text{in}~\Omega\times(0,T],\\u(x,0)=0,\quad\text{in}~\Omega,\\u=0,\quad\text{on}~\partial\Omega\times[0,T].\end{cases}
\end{equation}
In this case, the domain is set to be the 2-D square $\Omega=[-1,1]^2$, and the true solution is set to be
\begin{equation}
u_1(x,t) = \exp\left(\sin(2\pi wt)(x_1^2-1)(x_2^2-1)\right)-1,
\end{equation}
where $w$ is the wave number in the time variable.

First, we solve the equation with $T=1$ and $w=1,8$. For $w=1$, $u_1$ has only one wave in time and is hence less oscillatory. It is observed in Table \ref{Tab_case1_w1} that both DABG and DLS can find solutions with small errors, but DABG with errors $O(10^{-8})$ is clearly more accurate than DLS with errors $O(10^{-4})$. For $w=8$, the solution is highly oscillatory in time. It is observed in Table \ref{Tab_case1_w8} DABG can still obtain errors $O(10^{-4})$, yet DLS obtains larger errors $O(10^{-1})$. For DABG, the error curves versus $N$ are presented in Figure \ref{Fig_error_case1}, and the exponential decaying rate is observed for small $N$. When $N$ is relatively large, the error decaying of DABG is limited due to the optimization error of deep learning. Besides, it is notable that when $w=8$ the largest width choice $M=80$ leads to the worst result, probably because solving the optimization with wider networks is more difficult. We also show the temporal profiles of the true and numerical solutions at $(x_1,x_2)=(0,0)$ in Figure \ref{Fig_solution_case1}. When $w=8$, DABG can capture all 8 waves, while DLS can only recover 6 waves. This implies DABG is more capable of finding solutions with high frequencies in time.

Next, we investigate the error behavior versus time by solving the problem with $T=2,4,6,8,10$ for the fixed solution with $w=1$. We implement the method with $M=20$ and $N=20,40,60$. The errors are listed in Table \ref{Tab_case1_w1_vs_T} and the error curves versus time are shown in Figure \ref{Tab_case1_w1_vs_T}. We observe that for each fixed $N$, the error basically grows in exponential rate with $T$. Note the numerical result is consistent with the theory for $C^\infty$ solutions discussed in Section \ref{sec:long_time_estimates}.

\begin{table}
\centering
\subfloat[DABG]{
\begin{tabular}{|c|c|c|c|c|c|c|c|c|}
  \hline
   $N$ & $M=10$ & $M=20$ & $M=40$ & $M=80$\\\hline
4 & 7.145e-02 & 7.145e-02 & 7.145e-02 & 7.144e-02 \\\hline
6 & 3.604e-03 & 3.603e-03 & 3.597e-03 & 3.685e-03 \\\hline
8 & 5.015e-04 & 5.016e-04 & 5.000e-04 & 4.999e-04 \\\hline
10 & 4.531e-05 & 4.545e-05 & 4.539e-05 & 1.314e-04 \\\hline
12 & 1.749e-06 & 1.712e-06 & 1.804e-06 & 2.229e-06 \\\hline
14 & 1.813e-07 & 1.157e-07 & 5.795e-07 & 4.839e-07 \\\hline
16 & 1.517e-07 & \textbf{8.218e-08} & 1.840e-06 & 1.690e-07 \\\hline
18 & 1.607e-07 & 8.531e-08 & 3.663e-07 & 5.854e-07 \\\hline
\end{tabular}}\\
\subfloat[DLS]{
\begin{tabular}{|c|c|c|c|c|c|c|c|c|}
  \hline
   $L$ & $M=10$ & $M=20$ & $M=40$ & $M=80$\\\hline
3 & 5.907e-03 & 1.266e-03 & 1.911e-03 & 5.256e-03 \\\hline
4 & 2.898e-03 & \textbf{4.643e-04} & 1.246e-03 & 1.281e-03 \\\hline
5 & 2.571e-03 & 9.885e-01 & 4.881e-04 & 9.885e-01 \\\hline
\end{tabular}}
\caption{\em Relative $\ell^2$ errors of DABG and DLS methods in Case 1 with $w=1$.}\label{Tab_case1_w1}
\end{table}

\begin{table}
\centering
\subfloat[DABG]{
\begin{tabular}{|c|c|c|c|c|}
  \hline
   $N$ & $M=10$ & $M=20$ & $M=40$ & $M=80$ \\\hline
40 & 1.124e-01 & 1.124e-01 & 1.113e-01 & 1.076e-01 \\\hline
50 & 7.045e-03 & 7.045e-03 & 6.932e-03 & 2.397e-02 \\\hline
60 & 4.783e-04 & 4.791e-04 & 4.787e-04 & 1.628e-02 \\\hline
70 & 4.328e-04 & \textbf{4.321e-04} & 4.322e-04 & 1.526e-02 \\\hline
\end{tabular}}\\
\subfloat[DLS]{
\begin{tabular}{|c|c|c|c|c|}
  \hline
   $L$ & $M=10$ & $M=20$ & $M=40$ & $M=80$ \\\hline
3 & 9.607e-01 & 8.574e-01 & \textbf{4.019e-01} & 9.898e-01 \\\hline
4 & 8.675e-01 & 5.152e-01 & 5.973e-01 & 5.786e-01 \\\hline
5 & 9.960e-01 & 9.960e-01 & 9.960e-01 & 9.960e-01 \\\hline
\end{tabular}}
\caption{\em Relative $\ell^2$ errors of DABG and DLS methods in Case 1 with $w=8$.}\label{Tab_case1_w8}
\end{table}

\begin{figure}
\centering
\subfloat[$w=1$]{
\includegraphics[scale=0.6]{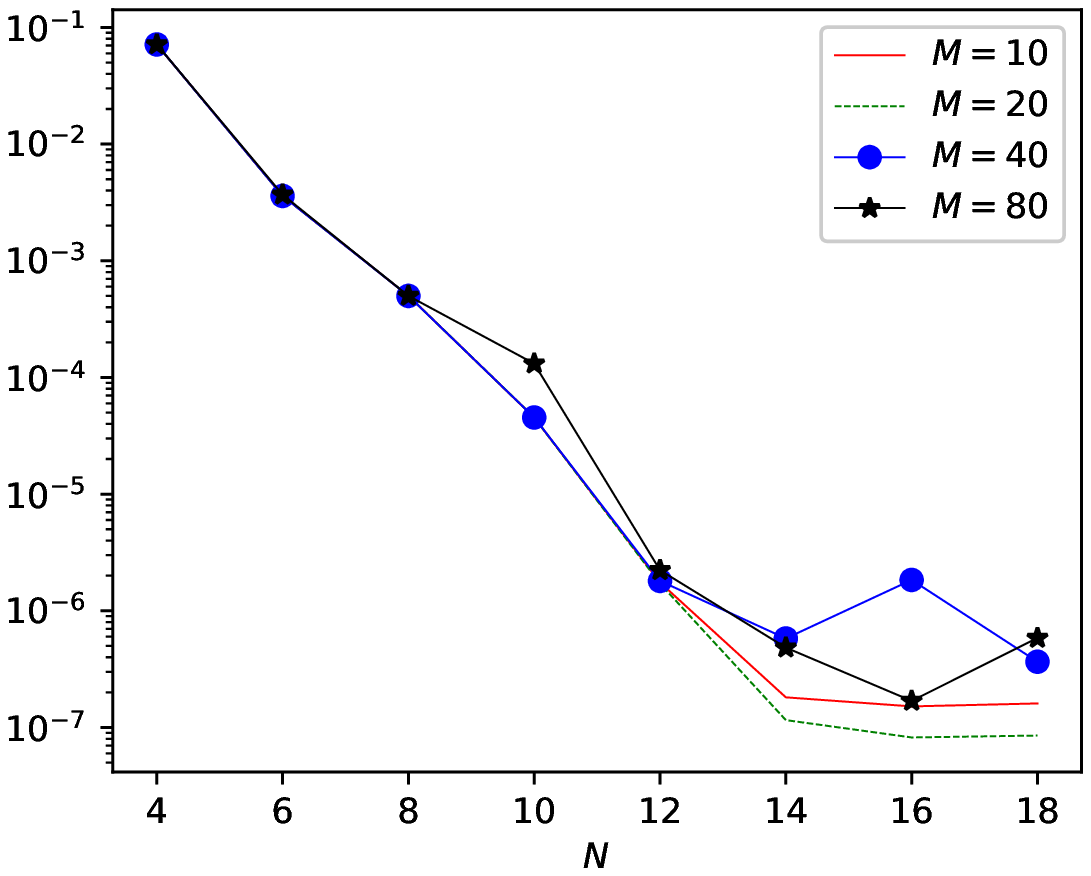}}
\subfloat[$w=8$]{
\includegraphics[scale=0.6]{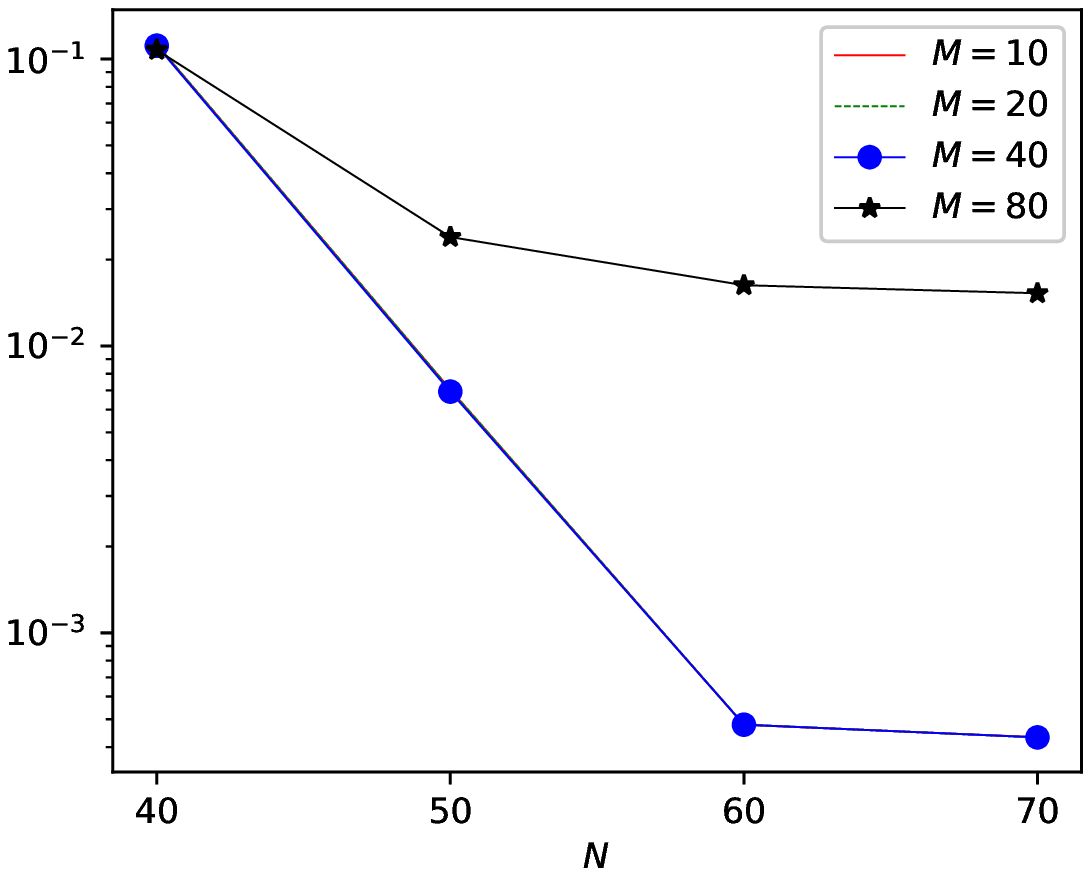}}
\caption{\em Relative $\ell^2$ errors versus $N$ of DABG in Case 1.}
\label{Fig_error_case1}
\end{figure}

\begin{figure}
\centering
\subfloat[$w=1$]{
\includegraphics[scale=0.6]{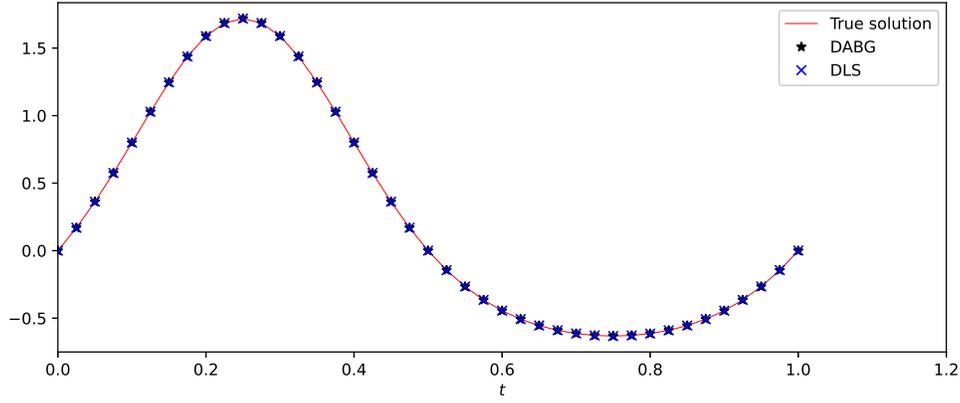}}\\
\subfloat[$w=8$]{
\includegraphics[scale=0.6]{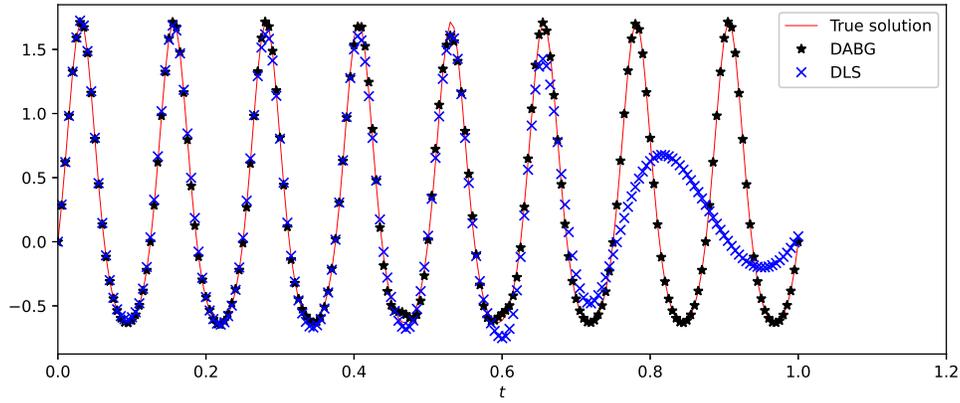}}
\caption{\em Temporal profiles of the true solution and numerical solutions obtained by DABG and deep least square (DLS).}
\label{Fig_solution_case1}
\end{figure}

\begin{table}
\centering
\begin{tabular}{|c|c|c|c|c|c|}
  \hline
   $N$ & $T=2$ & $T=4$ & $T=6$ & $T=8$ & $T=10$\\\hline
10 & 2.393e-02 & 7.501e-01 & 2.533e+00 & 3.727e+00 & 2.632e+00 \\\hline
20 & 2.590e-05 & 4.482e-02 & 1.836e-01 & 1.141e+00 & 2.213e+00 \\\hline
30 & 2.219e-06 & 1.394e-03 & 7.172e-02 & 1.179e-01 & 4.590e-01 \\\hline
40 & 2.497e-07 & 1.682e-05 & 1.362e-03 & 3.546e-02 & 1.020e-01 \\\hline
50 & 8.236e-06 & 1.379e-06 & 5.531e-04 & 3.683e-03 & 5.696e-02 \\\hline
60 & 4.682e-06 & 4.353e-07 & 2.176e-05 & 1.636e-03 & 1.574e-02 \\\hline
\end{tabular}
\caption{\em Relative $\ell^2$ errors of DABG methods in Case 1 for various $T$.}\label{Tab_case1_w1_vs_T}
\end{table}

\begin{figure}
\centering
\includegraphics[scale=0.6]{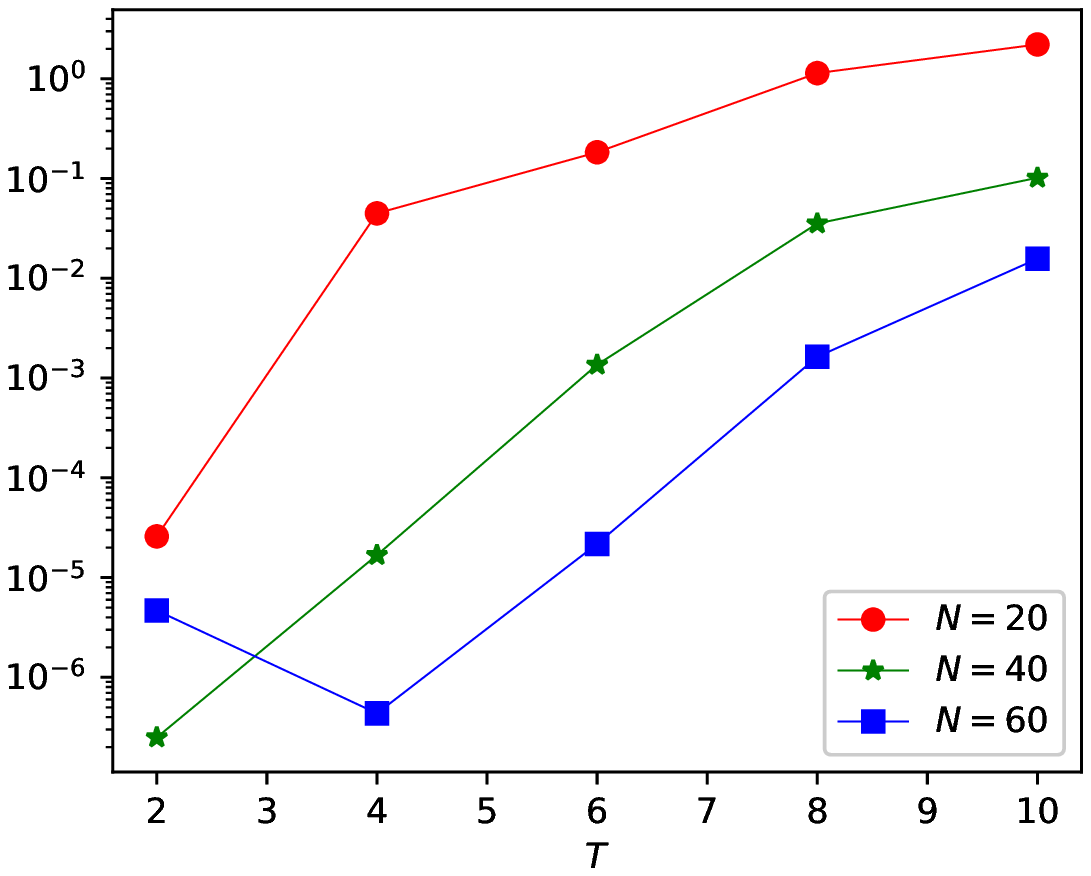}
\caption{\em Relative $\ell^2$ errors versus $T$ of DABG in Case 1.}
\label{Fig_error_vs_T_case1}
\end{figure}

\subsection{Case 2: 2-D heat equation oscillatory in time and space}
We continue solving the heat equation \eqref{heat_equ} with $\Omega:=[-1,1]^2$. The true solution is set as
\begin{equation}
u_2(x,t) = \exp\left(\sin(2\pi wt(x_1^2-1)(x_2^2-1))\right)-1.
\end{equation}
Note the wave number $w$ is coupled with both $t$ and $x$. We set $w=3$, and hence $u_2$ is oscillatory in both time and space. It is shown in Table \ref{Tab_case2_w3} that DABG with errors $O(10^{-3})$ still performs better than deep least square with errors $O(10^{-1})$. However, compared with the preceding case, it shows DABG is less accurate for solutions oscillatory in both time and space than those only oscillatory in time.

\begin{table}
\begin{minipage}[t]{0.47\linewidth}
\centering
\subfloat[DABG]{
\begin{tabular}{|c|c|c|c|}
  \hline
   $N$ & $M=20$ & $M=40$ & $M=80$ \\\hline
5 & 6.436e-01 & 5.612e-01 & 5.590e-01 \\\hline
10 & 5.184e-01 & 4.214e-02 & 3.742e-02 \\\hline
15 & 4.954e-01 & 2.647e-02 & 1.419e-02 \\\hline
20 & 5.902e-01 & 1.667e-02 & 3.095e-02 \\\hline
25 & 3.005e-01 & 1.577e-02 & \textbf{8.192e-03} \\\hline
\end{tabular}}\\
\subfloat[DLS]{
\begin{tabular}{|c|c|c|c|}
  \hline
   $L$ & $M=20$ & $M=40$ & $M=80$ \\\hline
3 & 8.103e-01 & \textbf{5.304e-01} & 7.790e-01 \\\hline
4 & 5.367e-01 & 2.288e+00 & 2.285e+00 \\\hline
5 & 2.291e+00 & 2.289e+00 & 2.285e+00 \\\hline
\end{tabular}}
\caption{\em Relative $\ell^2$ errors of DABG and DLS methods in Case 2 with $w=3$.}\label{Tab_case2_w3}
\end{minipage}
\hfill
\begin{minipage}[t]{0.47\linewidth}
\centering
\subfloat[DABG]{
\begin{tabular}{|c|c|c|c|}
  \hline
   $N$ & $M=40$ & $M=80$ & $M=160$ \\\hline
10 & 3.615e-01 & 3.615e-01 & 3.616e-01 \\\hline
15 & 3.077e-03 & 3.077e-03 & 3.078e-03 \\\hline
20 & 3.078e-03 & 3.078e-03 & 3.077e-03 \\\hline
25 & 3.075e-03 & 3.075e-03 & \textbf{3.074e-03} \\\hline
\end{tabular}}\\
\subfloat[DLS]{
\begin{tabular}{|c|c|c|c|}
  \hline
   $L$ & $M=40$ & $M=80$ & $M=160$ \\\hline
3 & 6.844e-01 & 9.935e-01 & 4.681e-02 \\\hline
4 & \textbf{4.542e-02} & 5.292e-01 & 9.935e-01 \\\hline
5 & 9.935e-01 & 9.935e-01 & 9.935e-01 \\\hline
\end{tabular}}
\caption{\em Relative $\ell^2$ errors of DABG and DLS methods in Case 3 with $w=3$.}\label{Tab_case3_w3}
\end{minipage}
\end{table}

\subsection{Case 3: 20-D parabolic equation}
In this case, we test the capability of the methods in high dimensions. Let us consider the following parabolic equation with non-constant coefficients
\begin{equation}
\begin{cases}u_t-\nabla\cdot(a(x)\nabla u) = f,\quad\text{in}~\Omega\times(0,T],\\u(x,0)=0,\quad\text{in}~\Omega,\\u=0,\quad\text{on}~\partial\Omega\times[0,T],\end{cases}
\end{equation}
with $a(x)=1+|x|^2/2$. The domain is set to be the 20-D unit ball $\Omega=\{x\in\mathbb{R}^{20}:|x|<1\}$, and the true solution is set to be
\begin{equation}
u_3(x,t) = \sin\left((\sin(2\pi wt)(|x|^2-1)\right).
\end{equation}
We set $w=3$, and obtained errors are shown in Table \ref{Tab_case3_w3}. It shows that DABG still obtains smaller errors ($O(10^{-3})$) than deep least square ($O(10^{-2})$).

\subsection{Case 4: 20-D Allen-Cahn equation}
In this case, we solve the Allen-Cahn equation
\begin{equation}
\begin{cases}u_t-\Delta u+u^3-u = f_\text{AC},\quad\text{in}~\Omega\times(0,T],\\u(x,0)=0,\quad\text{in}~\Omega,\\u=0,\quad\text{on}~\partial\Omega\times[0,T].\end{cases}
\end{equation}
The domain is set as $\Omega=\{x\in\mathbb{R}^{20}:|x|<1\}$, and the true solution is set to be
\begin{equation}
u_4(x,t) = \sin\left(\sin(2\pi wt)(|x|^2-1)\right).
\end{equation}
Due to the nonlinearity of the equation, we propose the following fixed point scheme,
\begin{equation}\label{51}
u^{(n)}_t-\Delta u^{(n)}-u^{(n)} = f_\text{AC}-(u^{(n-1)})^3,\quad n=1,2,\cdots,
\end{equation}
where $u^{(0)}$ is an initial guess, and $u^{(n)}$ is the numerical solution obtained in the $n$-th step. In practice, we use DABG to solve \eqref{41}. In each SGD iteration, we update the loss function according to \eqref{41}; Namely, in the $n$-th iteration, we compute the loss function with $f=f_\text{AC}-(u^{(n-1)})^3$, and update the solution by gradient descent. Obtained errors for $w=3$ are shown in Table \ref{Tab_case4_w3}. It shows that DABG (with errors $O(10^{-3})$) performs better than least square (with errors $O(10^{-2})$).

\begin{table}
\begin{minipage}[t]{0.47\linewidth}
\centering
\subfloat[DABG]{
\begin{tabular}{|c|c|c|c|}
  \hline
   $N$ & $M=40$ & $M=80$ & $M=160$ \\\hline
10 & 3.080e-01 & 3.080e-01 & 3.080e-01 \\\hline
14 & 2.300e-03 & 2.300e-03 & 2.300e-03 \\\hline
18 & \textbf{2.069e-03} & 2.069e-03 & 2.069e-03 \\\hline
22 & 2.070e-03 & 2.070e-03 & 2.070e-03 \\\hline
\end{tabular}}\\
\subfloat[DLS]{
\begin{tabular}{|c|c|c|c|}
  \hline
   $L$ & $M=40$ & $M=80$ & $M=160$ \\\hline
3 & 5.312e-01 & 9.937e-01 & 9.394e-01 \\\hline
4 & \textbf{7.791e-02} & 1.329e-01 & 9.937e-01 \\\hline
5 & 9.937e-01 & 9.937e-01 & 9.937e-01 \\\hline
\end{tabular}}
\caption{\em Relative $\ell^2$ errors of DABG and DLS methods in Case 4 with $w=3$.}\label{Tab_case4_w3}
\end{minipage}
\hfill
\begin{minipage}[t]{0.47\linewidth}
\centering
\subfloat[DABG]{
\begin{tabular}{|c|c|c|c|}
  \hline
   $N$ & $M=40$ & $M=80$ & $M=160$ \\\hline
10 & 3.055e-01 & 3.057e-01 & 3.056e-01 \\\hline
14 & 3.505e-03 & 3.505e-03 & 3.513e-03 \\\hline
18 & \textbf{1.907e-03} & 1.907e-03 & 1.913e-03 \\\hline
22 & 2.114e-03 & 2.114e-03 & 2.121e-03 \\\hline
\end{tabular}}\\
\subfloat[DLS]{
\begin{tabular}{|c|c|c|c|}
  \hline
   $L$ & $M=40$ & $M=80$ & $M=160$ \\\hline
3 & 1.141e-01 & 1.142e+00 & 9.744e-01 \\\hline
4 & \textbf{8.465e-02} & 9.745e-01 & 9.744e-01 \\\hline
5 & 9.745e-01 & 9.745e-01 & 9.745e-01 \\\hline
\end{tabular}}
\caption{\em Relative $\ell^2$ errors of DABG and DLS methods in Case 5 with $w=3$.}\label{Tab_case5_w3}
\end{minipage}
\end{table}

\subsection{Case 5: 20-D hyperbolic equation}\label{sec_case5}
In this case, we solve the following hyperbolic equation with non-constant coefficients,
\begin{equation}
\begin{cases}u_{tt}-\nabla\cdot(a(x)\nabla u) = f,\quad\text{in}~\Omega\times(0,T],\\u(x,0)=u_t(x,0)=0,\quad\text{in}~\Omega,\\u=0,\quad\text{on}~\partial\Omega\times[0,T],\end{cases}
\end{equation}
with $a(x)=1+|x|^2/2$. The domain $\Omega=\{x\in\mathbb{R}^{20}:|x|<1\}$ and the true solution is
\begin{equation}
u_5(x,t) = \sin\left(t\sin(2\pi wt)(|x|^2-1)\right).
\end{equation}

We report the errors for $w=3$ in Table \ref{Tab_case5_w3}. Same as the previous parabolic cases, DABG obtains smaller errors than deep least square in this hyperbolic example.

\section{Conclusion}
We develop a novel and highly accurate DABG method for solving
high-dimensional evolution equations. In this method, the solution is approximated by the adaptive basis structure, which is formulated as a tensor-product of a set of orthogonal polynomials and a class of DNNs. The first step is the time discretization of the equation using traditional Galerkin formulation, leading to a linear system with unknown DNNs. The second step is determining these DNNs via the minimization of the residual of the linear system. This method enjoys not only the high accuracy of Galerkin formulation, but also the capability of high-dimensional approximation of DNNs. For problems with oscillatory solutions, the numerical results show significant advantage of the proposed method over the popular deep least-squares method.

Note that the technique for time discretization is not limited to the Galerkin method with polynomial-type basis used in this work. For some special time-dependent problems (e.g., porous medium equation and Navier-Stokes equations), it is promising to consider other techniques such as finite element approximation and energy stable numerical schemes, which can be further studied. Another future direction is to consider
better deep solvers for the spatial equation after time discretization. In this work, we formulate the loss function following the idea of minimal residual, which is straightforward but sometimes less efficient. The simultaneous training of all the unknown DNNs is also computationally expensive. Since the coefficient matrix has a sparse banded structure, faster learning relying on this property is expected to be developed.

\bibliography{expbib}
\bibliographystyle{plain}

\appendix

\end{document}